\documentclass[12pt]{amsart} 
\usepackage{amsmath, newcommands2006, amsfonts, epic, amscd, amssymb, array,bm}

\evensidemargin=\oddsidemargin
\newtheorem{thm}{Theorem}[section] 
\newtheorem{proposition}[thm]{Proposition}
\newtheorem{lemma}[thm]{Lemma}  
\newtheorem{corollary}[thm]{Corollary}
\theoremstyle{definition}  
\newtheorem{example}[thm]{Example}
\newtheorem{definition}[thm]{Definition} 
\theoremstyle{remark}  
\newtheorem{remark}[thm]{Remark}
\newtheorem{notation}[thm]{Notation}

\newcommand{\pr}{p}

\begin{document} 

\title{Anosov  Automorphisms of Nilpotent Lie Algebras}
\author{Tracy L. Payne} 
\address{Department of Mathematics,  Idaho State University, 
Pocatello, ID 83209-8085} 
\email{payntrac@isu.edu}


\subjclass{Primary: 22E25, 22D45, 37D20}

 \begin{abstract} Each  matrix $A$ in $GL_n(\boldZ)$
naturally defines an automorphism $f$ of the  free $r$-step nilpotent
Lie algebra $\frakf_{n,r}.$  We study the relationship between the
matrix $A$ and the  eigenvalues and rational invariant subspaces for
$f.$   We give applications to the study of Anosov automorphisms.
\end{abstract}

\keywords{Anosov automorphism,  Anosov diffeomorphism, Anosov Lie
algebra,  hyperbolic automorphism,  nilmanifold, nilpotent Lie
algebra}

 \maketitle




\section{Introduction}\label{introduction}

\subsection{Anosov maps} Anosov maps are fundamental objects 
in the field of dynamical systems.  A $C^1$ diffeomorphism $f$ of a
compact Riemannian manifold $M$ is called \textit{Anosov} if there
exist constants $\lambda$ in $(0,1)$ and $c > 0$ along with  a
$df$-invariant splitting $TM = E^s \oplus E^u$  of the tangent  bundle
of $M$ so that for all $n \ge 0$,
\begin{align*} \| df^n_x\bfv\| &\le c \lambda^n \| \bfv \| \quad
\text{for all $\bfv$ in $E^s(x)$, and} \\ \| df^{-n}_x\bfv \| &\le c
\lambda^n \| \bfv \| \quad \text{for all $\bfv$ in
$E^u(x)$.} \label{expanding definition}
\end{align*}  The standard example of an Anosov map is a toral map
defined by a unimodular hyperbolic automorphism  of $\boldR^n$ that
preserves an integer lattice. 

The only other known examples of Anosov maps arise from  automorphisms
of nilpotent groups.   A hyperbolic automorphism of a simply connected
nilpotent Lie group $N$ that fixes a torsion-free lattice $\Gamma < N$
descends to an Anosov diffeomorphism of the compact nilmanifold
$N/\Gamma$.  It is also possible that such an $N/\Gamma$ has a finite
quotient, called an  {\em infranilmanifold}, and that  the Anosov map
on $N/\Gamma$  finitely covers  an Anosov map of the infranilmanifold.  

An automorphism of a Lie algebra that descends to an Anosov map of a
compact quotient of the corresponding  Lie group is called an {\em
Anosov automorphism.}  Nilpotent Lie algebras are the only Lie algebras
that admit Anosov  automorphisms.  A Lie algebra $\frakn$ is called
{\em Anosov} if there exists  a  basis $\calB$  for $\frakn$ with
rational structure constants,  and  there exists a   hyperbolic
automorphism  $f$ of $\frakn$  with respect to which $f$ is
represented relative to $\calB$ by  a matrix in $GL_n(\boldZ).$  A
simply connected Lie group admits a hyperbolic automorphism preserving
a lattice if and only if its Lie algebra is Anosov
(\cite{auslander-scheuneman}). 

 In this paper we study the properties of Anosov automorphisms and
Anosov Lie algebras.  There has already been  some  progress in this area.  
S. G. Dani showed that the free $r$-step nilpotent Lie algebra
$\frakf_{n,r}$ on $n$ generators admits an Anosov automorphism when $r<n$
(\cite{dani-78}). Dani and Mainkar considered  when 
 two-step nilpotent Lie algebras defined by graphs 
admit  Anosov automorphisms (\cite{dani-mainkar-05}).  Real Anosov Lie
algebras and all their rational forms
have been classified in dimension eight and less (\cite{ito-89},
\cite{lauret-will-05}).  
 Lauret observed that
the classification problem for Anosov Lie algebras
contains within it the problem of classifying all Lie algebras
admitting $\boldZ^+$ derivations (\cite{lauret-03c}, \cite{lauret-03}).

  Auslander and Scheuneman established the correspondence between
Anosov automorphisms of nilpotent Lie algebras and semisimple
hyperbolic automorphisms of free nilpotent Lie algebras preserving
ideals of a certain type (\cite{auslander-scheuneman}).   A matrix
$A$ in $GL_n(\boldZ),$ together with a rational basis $\calB$ of 
$\frakf_{n,r}$ induces 
an automorphism $f^A$ of $\frakf_{n,r}.$ Suppose that
$\fraki$ is an ideal of $\frakf_{n,r}$ such that 
\begin{enumerate} 
\item{$\fraki$ is invariant under ${f^A},$}\label{prop0}
\item{the restriction of ${f^A}$ to $\fraki$ is unimodular,}
\item{$\fraki$ has a basis that consists  of $\boldZ$-linear
combinations of elements of $\calB$, and}\label{prop1}
\item{ all eigenspaces for ${f^A}$ for eigenvalues with modulus one
are contained in $\fraki$.}\label{prop2}
\end{enumerate} 
 If we  let $\frakn =  \frakf_{n,r} / \fraki$ and
let  $\pr : \frakf_{n,r} \to \frakn$ be the projection map,  there is
an  Anosov automorphism $\overline{f} : \frakn \to \frakn$ such that
$\overline{f} \pr = \pr f^A$.   We will  call the four conditions 
the {\em Auslander-Scheuneman conditions.}
Auslander and Scheuneman showed that 
any semisimple Anosov automorphism $f$ of an $r$-step
nilpotent Lie algebra $\frakn$ may be represented in the manner just
described, relative to a rational basis $\calB$ of a free nilpotent Lie algebra
$\frakf_{n,r}$, a semisimple matrix $A$ in $GL_n(\boldZ),$ 
and an ideal $\fraki$ in $\frakf_{n,r}$ satisfying the 
four conditions.  We will always assume without loss of generality that 
$\fraki < [\frakf_{n,r},\frakf_{n,r}].$

 In order to
understand general properties of Anosov Lie algebras, one must understand
the  kinds of  ideals of  free nilpotent Lie algebras
that satisfy the Auslander-Scheuneman conditions
 for some automorphism $f^A$ defined by a  matrix $A \in GL_n(\boldZ).$
  The dynamical properties of a 
toral Anosov automorphism of $\boldR^n/\boldZ^n$ 
are closely related to the algebraic properties
of the characteristic polynomial $p$ of the matrix $A$ in 
$GL_n(\boldZ)$ used to define the automorphism (See \cite{everest-ward}).
We show in this work that, similarly, the algebraic properties of the
  characteristic polynomial $p$ of the matrix $A$ defining the
automorphism 
$f^A$ of a free nilpotent Lie algebra
determine  the structure of the ideals that  satisfy the
Auslander-Scheuneman conditions for $f^A.$  

\subsection{Summary of results} Now we summarize the main ideas of the
paper. 
 We associate to any  automorphism $f^A, 
A \in GL_n(\boldZ),$ of a free  $r$-step nilpotent Lie algebra
$\frakf_{n,r}$  an $r$-tuple of  polynomials $(p_1, p_2, \ldots, p_r).$  
For $i=1, \ldots,r,$ the
polynomial $p_i$ is the characteristic polynomial of a matrix representing
the automorphism $f^A$
 on an $i$th step $V_i$ of $\frakf_{n,r}.$  
Let $K$ denote the splitting field of $p_1$ over $\boldQ,$
and let $G$ denote the Galois group for $K$ over $\boldQ.$ 
 We  associate  to the
automorphism  the action of the finite group $G$ on $\frakf_{n,r}(K),$
the free nilpotent 
Lie algebra over $K.$   We show in Theorem 
\ref{actions} that  $G$ orbits  in $\frakf_{n,r}(K)$ correspond to
rational invariant subspaces for $f^A,$ and the 
characteristic polynomial for the restriction of $f^A$ to such an invariant
subspace is a power of an irreducible polynomial.

We analyze Anosov Lie algebras using the following general approach. 
 We fix a free
$r$-step nilpotent Lie algebra $\frakf_{n,r}.$  We consider the
class of automorphisms of $\frakf_{n,r}$ whose associated polynomial
$p_1$ has Galois group $G,$ where $G$ is isomorphic to a subgroup of
the symmetric group $S_n.$
We let $(p_1, p_2, \ldots, p_r)$ be the $r$-tuple of polynomials
associated to such a polynomial $p_1.$
Our goal is to determine the factorizations of the polynomials
$p_1, \ldots, p_r;$ this will tell us what rational invariant subspaces
for $f$ are.   Such subspaces generate any ideal satisfying the 
Auslander-Scheuneman conditions.  
First we analyze the factorizations of $p_2, \ldots, p_r$ into
powers of irreducibles by understanding 
orbits of the action of the Galois group of $p_1$ on $\frakf_{n,r}(K).$  
 Then we determine whether the corresponding 
rational invariant subspaces are minimal and whether there are eignevalues 
of modulus one 
using  ideas from number theory (See Proposition \ref{full rank} and 
Lemma \ref{how it can factor}).    

    We extend the  classification of
Anosov Lie algebras to some new classes of two-step Lie
algebras. 
\begin{thm}\label{main}   Suppose that $\frakn$ is a 
two-step Anosov
Lie algebra of type $(n_1, n_2)$ 
with associated polynomials $(p_1,p_2).$   
 Let $G$ denote the Galois group of $p_1.$
\begin{enumerate}
\item{If 
 $n_1 = 3, 4$ or $5,$  then $\frakn$ is one of the
Anosov  Lie algebras listed in Table \ref{r=2}.}\label{classify-lowdim}
\item{If  $p_1$ is irreducible and
 the action of $G$ on the roots of $p_1$ is doubly transitive, 
then $\frakn$ is isomorphic to the free nilpotent Lie algebra
$\frakf_{n,2}$. }\label{classify-full rank}
\end{enumerate}
\end{thm}
 We can also classify Anosov Lie
algebras  admitting automorphisms whose polynomials $p_1$ 
have certain specified  Galois groups.  
\begin{thm}\label{Sn-Cn} 
Let $\overline{f}$ be a semisimple Anosov automorphism of 
an $r$-step Anosov Lie algebra.
Let $(p_1, \ldots, p_r)$ be the $r$-tuple of polynomials associated
to the  automorphism $f$ of the free nilpotent Lie algebra
 $\frakf_{n,r}$ induced by $\overline{f}.$    Suppose that $p_1$ is irreducible.
\begin{enumerate}
\item{If  the  polynomial $p_1$ is of prime degree with
cyclic Galois group, then $\frakn$  is one of the Lie algebras 
of  type $C_n$ defined over $\boldR$
 as in Definition \ref{ideal-i}.  Conversely,
if $n$ is prime, and $\fraki$ is an ideal 
of $\frakf_{n,r}$ of cyclic type defined over $\boldR$ containing  the 
ideal $\frakj_{n,r}$ defined in Definition \ref{j-def}, then the 
 Lie algebra $\frakn = \frakf_{n,r}/\fraki$ 
 is Anosov.  }\label{classify-prime}
\item{If the Galois group of $p_1$ is 
symmetric,  then
\begin{enumerate}
\item{If $r = 2,$  then $\frakn$ is isomorphic to
$\frakf_{n,2},$}
\item{If  $r = 3,$  then $\frakn$ is isomorphic to one of 
the following five Lie algebras: 
$\frakf_{n,3},$ $\frakf_{n,3}/F_1, \frakf_{n,3}/F_2,$
 $\frakf_{n,3}/(F_1 \oplus F_{2a}),$  and $\frakf_{n,3}/F_{2a},$ 
where the ideals  $F_1$ and  $F_2$ are as defined in Equation  
\eqref{F defs} of Section \ref{p2p3} and $F_{2a}$ is as 
in Proposition \ref{Fdecomp}. }
\end{enumerate}}
\end{enumerate}
\end{thm} 
 Matrices in $GL_n(\boldZ)$ having characteristic
polynomial with symmetric Galois group are dense in the sense of
thick and thin sets (\cite{serre-92}); hence, the second part of 
the previous theorem describes 
Anosov automorphisms of two- and three-step Lie algebras 
that are generic in this sense. 

We investigate some general properties of Anosov
automorphisms.  We can describe the dimensions of
 minimal nontrivial rational invariant subspaces.   
Out of  such analyses we obtain the following special  case.
\begin{thm}\label{general-properties} Suppose that $\frakn$ is an
Anosov Lie algebra of type  $(n_1, \ldots, n_r).$   If $n_1=3,$  
then $n_i$ is a multiple of $3$ for all $i=2, \ldots, r,$ and if $n_1=4,$ 
then $n_i$ is even for all $i=2, \ldots, r.$  
If  $n_1$ is prime and
 the polynomial $p_1$ is irreducible,  
then $n_1$ divides $n_i$ for all $i = 2, \ldots,r < n.$
\end{thm}
The results of \cite{mainkar-06b} give 
an alternate proof of this theorem. 

One way to approach the classification problem is to fix the field in
which the spectrum of an Anosov automorphism 
lies.  The following theorem describes all
Anosov automorphisms whose spectrum lies in a quadratic extension of
$\boldQ.$
\begin{thm}\label{spectrum}  
Let $f$ be a semisimple Anosov automorphism of a two-step nilpotent Lie
algebra $\frakn.$  Let $\Lambda \subset \boldR$ denote the spectrum of $f,$
and let $K$ denote the finite extension $\boldQ(\Lambda)$  of $\boldQ.$
If $K$  is a quadratic extension of $\boldQ,$ 
then  $\frakn$ is one of the Anosov
Lie algebras defined in Definition \ref{graph}.  
\end{thm} 
The paper is organized as follows. 
 In Section \ref{preliminaries},
 we review background material on nilpotent Lie algebras, 
Anosov automorphisms and algebraic numbers, and  we define the
$r$-tuple of polynomials associated to an Anosov automorphism
of a Lie algebra.  In Section \ref{polynomials}, we describe properties
of the $r$-tuple of polynomials,  such
as their reducibility and their Galois groups. In Proposition \ref{full rank},
we consider the set of roots of an Anosov polynomial
and  describe  multiplicative relationships among them; this 
number-theoretic result may be
interesting in its own right. 
In Section \ref{action-section},  we associate  to an
automorphism $f$ of a free nilpotent Lie algebra  $\frakf_{n,r}$
 the action of a Galois group $G,$ and in Theorem \ref{actions} we
relate rational invariant subspaces of   $\frakf_{n,r}$ to the orbits of $G.$   In Section \ref{symmetric-cyclic} we consider
Anosov Lie algebras for which the associated
 Galois group  is symmetric or cyclic.
Finally, in Section \ref{2&3}, we apply the results from
 previous sections to the problem of classification
of Anosov Lie algebras whose associated polynomial $p_1$ 
has   small degree.  
Although the  theorems we have stated above follow from various
results distributed throughout  this work, for the sake of clarity, 
in  Section \ref{summary} we provide
self-contained proofs of the theorems.

 \thanks{ This
project was initiated through discussions with Ralf Spatzier  during a
visit to the University of Michigan funded by the University of
Michigan's NSF Advance grant \#0123571.  The author is  grateful to Ralf
Spatzier for  his interest and encouragement, and many useful questions and 
comments.  Peter Trapa, Roger Alperin and Raz Stowe
also provided helpful input.  }

\section{Preliminaries}\label{preliminaries} In this section, we
describe the  structure of free nilpotent Lie algebras and their
automorphisms, and we review
some concepts from number theory that we will use later.  We conclude with
some examples to illustrate the concepts presented. 

\subsection{Nilpotent Lie algebras}\label{free-section} Let $\frakn$
be a Lie algebra defined over field $K.$  
The \textit{central  descending series} for
$\frakn$ is defined by $\frakn^0 = \frakn,$  and $\frakn^i =
[\frakn,\frakn^{i-1}]$ for $i \ge 1$.  If $\frakn^r = 0$ and
$\frakn^{r-1} \ne 0,$ then $\frakn$ is said to be  $r$-\textit{step
nilpotent}.  When $\frakn$ is a nilpotent Lie 
algebra defined over field $K$ and $n_i$ is the dimension of
the vector space $\frakn^i / \frakn^{i - 1}$
over $K,$ then $(n_1, n_2, \ldots, n_r)$  is called
the \textit{type} of $\frakn$.  

The \textit{free $r$-step nilpotent Lie algebra on $n$ generators over the field
$K$,} denoted
$\frakf_{n,r}(K),$ is defined to be the quotient  algebra $\frakf_n(K) /
\frakf_n^{r + 1}(K),$ where  $\frakf_n(K)$ is the free nilpotent Lie algebra
on $n$ generators over $K.$  Given  a  set $\calB_1$
of $n$ generators, the free nilpotent Lie algebra  $\frakf_{n,r}(K)$
can be written as the direct sum $V_1(K) \oplus \cdots \oplus V_r(K),$ where
$V_1(K)$ is defined to be the span of $\calB_1$ 
over $K$ and for $i = 2, \ldots, r,$  the
subspace $V_i(K)$ is defined to be the span over $K$ of 
$i$-fold brackets of the generators.  We will call the space 
$V_i(K)$ the {\em $i$th step} of $\frakf_{n,r}(K)$ without always
explicitly mentioning the dependence on $\calB_1.$  
When the field $K$ has characteristic zero, we 
 identify the prime subfield of $K$ with $\boldQ.$ 
For our purposes, fields  that  we consider
 will be intermediate to $\boldQ$ and $\boldC:$
 one of $\boldQ, \boldR, \boldC,$ or the splitting field for a 
polynomial in $\boldZ[x].$ 
 We will always
assume that a generating set $\calB_1$ for
a free nilpotent Lie algebra $\frakf_{n,r}(K)$ has cardinality $n.$

The most natural basis to use for a free nilpotent Lie algebra 
is a Hall basis.  Let $\bfx_1, \bfx_2, \ldots, \bfx_n$ be $n$ generators
for $\frakf_{n,r}(K)$.   We call these the {\em standard monomials of
degree one.}  \textit{Standard monomials of degree} $n$ are defined
inductively:     After the monomials of degree $n-1$ and less  have
been defined, we define an order relation $<$ on them so that if
$\degree u < \degree v,$  then $u < v.$ Any linear combination  of
monomials of degree $i$ will be said to be of degree $i.$ If $u$ has
degree $i$ and $v$ has degree $j,$ and $i+j=k,$  we define $[u,v]$ to
be a standard monomial of degree $k$ if  $u$ and $v$ are standard
monomials and $u > v,$ and if $u = [x,y]$ is the form of the standard
monomial $u,$ then $v \ge y.$ The standard monomials of degree $r$ or
less form a basis  for $\frakf_{n,r}(K),$ called the {\em Hall basis}
(\cite{hall50}).    For $i=1, \ldots, r,$ the subset $\calB_i = \calB
\cap V_i(K)$ of the basis $\calB$ is a basis for the  $i$th step
$V_i(K)$ of $\frakf_{n,r}(K)$ consisting of 
elements of the Hall basis 
of degree $i.$  To each monomial of degree $i$ we can also associate a
{\em Hall word} of length $i$  from a given alphabet 
 $\alpha_1, \ldots, \alpha_n$ of
$n$ letters; for example, $[[\bfx_3,\bfx_1], \bfx_2]$  becomes
the word $\alpha_3 \alpha_1 \alpha_2.$ 

Suppose $\frakg$ is a Lie algebra defined over a field  
$K$ of characteristic zero.  Suppose that $\calB$ is a basis of
$\frakg$ having rational structure constants.  The basis $\calB$ determines
a {\em rational structure} on $\frakg.$ A subspace $E$ of $\frakg$ spanned by
$\boldQ$-linear combinations of elements of $\calB$ is called  a {\em
rational subspace} for this rational structure.
 Since the structure constants for the free  nilpotent 
Lie algebra $\frakf_{n,r}(K)$ relative to 
a Hall basis $\calB$ are rational,  a
Hall basis  $\calB$ for $\frakf_{n,r}(K)$ defines a rational structure
on  $\frakf_{n,r}(K).$ 

\begin{example}\label{2,3-Hall words} Let $\calC_1 = \{
\bfz_i\}_{i=1}^n$ be a set of $n$ generators for the  free $r$-step
nilpotent Lie algebra $\frakf_{n,r}(K)$ 
on $n$ generators over a field $K$ 
and let  $\calC =
\cup_{i=1}^r \calC_i$ be the Hall basis determined by $\calC_1.$
Elements of  $ \calC_2,$  where  $r \ge 2,$  are of
the form $[\bfz_i,\bfz_j]$ with $i > j,$ hence the dimension of
$V_2(K)$ over $K$ is $  (\begin{smallmatrix} n
\\ 2 \end{smallmatrix}).$ 
When  $r \ge 3,$ from the definition of Hall monomial, elements in the
set $ \calC_3$ for $\frakf_{n,r}(K)$  are of the form $[[\bfz_i,\bfz_j],
\bfz_i]$ or  $[[\bfz_i,\bfz_j],
\bfz_j]$ with $i > j$ or, if $n \ge 3,$ of the form
$[[\bfz_i,\bfz_j],\bfz_k]$ with $i,j,k$ distinct and $i$ and $k$
greater than $j.$  There are  $ n(n-1)$ standard Hall monomials of the
first type, and when $n \ge 3,$ there are $2  (\begin{smallmatrix} n
\\ 3 \end{smallmatrix})$ standard  Hall monomials of the second type,
for a dimensional total  of $\smallfrac{1}{3}(n+1)n(n-1)$ for the
third step $V_3(K)$ of $\frakf_{n,r}(K).$   We let
$ \calC_3^\prime$ denote the set of   standard Hall monomials of the
first type, and let  $ \calC_3^{\prime \prime}$ denote the set of
standard Hall monomials of the second type:
\begin{align*} 
\calC_3^\prime &= \cup_{1 \le j < i  \le n} \{[[\bfz_i,\bfz_j],\bfz_i],
[[\bfz_i,\bfz_j],\bfz_j]   \}, 
   \quad \text{and}\\ 
\calC_3^{\prime \prime} &= \cup_{1 \le j < i < k \le n} \{[[\bfz_i,\bfz_j],\bfz_k],
[[\bfz_k,\bfz_j],\bfz_i]  \}.
\end{align*}  Define subspaces $F_1(K)$ and $F_2(K)$ of $\frakf_{n,r}(K)$ by 
\begin{equation}
\label{F defs} 
F_1(K) = \myspan_K  \calC_3^\prime, \qquad \text{and} 
\quad F_2(K) = \myspan_K  \calC_3^{\prime
\prime}. \end{equation} The subspace $V_3(K)$ is the direct sum of
$F_1(K)$ and $F_2(K)$   since $\calC_3$ spans $V_3(K)$ and is the
disjoint union of $\calC_3^\prime$ and $\calC_3^{\prime \prime}.$
\end{example}

\subsection{Anosov automorphisms}\label{setup}

As we discussed previously, every Anosov automorphism can be represented
in terms of a matrix $A$ in $GL_n(\boldZ),$   an automorphism $f^A$
of $\frakf_{n,r}$ induced by $A$ and an ideal $\fraki < \frakf_{n,r}$
satisfying the four Auslander-Scheuneman conditions.  In this section
we spell out some of the details involved in such a representation. 

Let
 $\frakf_{n,r}(K) = \oplus_{i=1}^r V_i(K)$ be the free  $r$-step
nilpotent Lie algebra over field $K$ with
 a set $\calB_1$ of $n$  generators.  
Let $\calB = \cup_{i=1}^r \calB_i$ be the Hall basis
 determined by $\calB_1.$  
Let $A$ be a  matrix in
$GL_{n}(\boldZ)$ having no eigenvalues of modulus one. Together 
the matrix $A$ and 
and the basis $\calB_1$  define a linear map $f_1: V_1(K) \to V_1(K).$  The
map $f_1$ induces an automorphism $ f^A_K$ 
of  $\frakf_{n,r}(K)$ that when
restricted to $V_1(K)$ equals $f_1.$  For all  $i=1, \ldots, r,$ the
restriction $f_i$ of $f^A_K$ to $V_i(K)$ can be  represented with respect
to the basis $\calB_i$ of $V_i(K)$ by a matrix $A_i$ having integer
entries that are independent of the field $K.$   

For $i=1, \ldots, r,$ let $p_i$ denote the characteristic
polynomial of $A_i.$ We define the {\em $r$-tuple of polynomials
associated to $f$} to be $(p_1, \ldots, p_r).$  Note that
 all of the polynomials are monic with integer coefficients, and
 there is no dependence on $K$ in defining the polynomials: either
the matrix $A$ or the polynomial $p_1$ alone is enough to  
uniquely define the $r$-tuple $(p_1, \ldots, p_r).$ 

A Lie algebra admits an Anosov automorphism  if and only if it  admits
a semisimple Anosov automorphism  (\cite{auslander-scheuneman}). 
 Assume
that the linear map $f_1: V_1(L) \to V_1(L)$ defined by 
$A \in GL_n(\boldZ)$ and rational basis $\calB$
is diagonalizable over the field $L$ (where $\mychar L = 0$). 
The vector space $V_1(L)$
 can be  decomposed into the direct sum of minimal
nontrivial rational $f_1$-invariant  subspaces $E_1, \ldots, E_s$.
For each rational invariant subspace $E_j, j = 1, \ldots, s,$
the restriction of  $f_1$ to $E_j$ is diagonalizable over $L.$
Hence there is a basis  $\calC_1 = \{\bfz_1, \ldots, \bfz_n\}$  of 
$V_1(L)$ consisting of eigenvectors of $f_1$
 with each eigenvector properly contained in
one of  the  subspaces $E_1, \ldots, E_s.$  Let
$\calC$ be the Hall basis of $\frakf_{n,r}(L)$ determined by $\calC_1.$
We will call such an eigenvector basis for an automorphism 
$f$ of   $\frakf_{n,r}(L)$  {\em compatible with  the
rational structure}, and in the future, when we use 
eigenvector bases for free nilpotent  Lie algebras we will always
choose them to be compatible with the rational structure determined
by a fixed Hall basis.  

\begin{notation}
 We shall use $\calB$ to
denote the Hall basis of a free nilpotent Lie algebra 
$\frakf_{n,r}(K)$ that determines the
rational structure and that with a matrix in $GL_n(\boldZ)$
defines the Anosov automorphism,
 while we will use $\calC$ to denote a Hall basis that 
diagonalizes the Anosov automorphism. 

Suppose that $K$ has characteristic zero
and $\calB$ is a fixed Hall basis of $\frakf_{n,r}(K),$ and identify the 
  prime subfield of $K$ with $\boldQ.$ 
We will use $\frakf_{n,r}(\boldQ)$
to denote the subset of $\frakf_{n,r}(K)$ that is 
the   $\boldQ$-span   of $\calB$ in  $\frakf_{n,r}(K).$ 
\end{notation}

At times we will move between free nilpotent Lie algebras $\frakf_{n,r}(K)$
and  $\frakf_{n,r}(L)$ defined over different field extensions 
$K$ and $L$ of $\boldQ.$  We define a correspondence
between rational  $f^A_K$-invariant subspaces of  $\frakf_{n,r}(K)$
and rational  $f^A_L$-invariant subspaces of  $\frakf_{n,r}(L):$
\begin{definition}\label{correspond}
Let $K$ and $L$ be fields with characteristic zero.
 Let  $\calB_1(K)$  and  $\calB_1(L)$ be generating sets,  
both of cardinality $n,$ for free nilpotent Lie algebras
 $\frakf_{n,r}(K)$ and $\frakf_{n,r}(L)$
respectively, and let  $\calB(K)$  and  $\calB(L)$ be the Hall bases 
defined by  $\calB_1(K)$  and  $\calB_1(L)$ respectively.  
  A bijection $i_1: \calB_1(K) \to \calB_1(L)$  of the generating sets  
naturally induces a bijection $i: \calB(K) \to \calB(L)$
 of the Hall bases, and this in turn defines an isomorphism 
 $\overline{i}$ from  $\frakf_{n,r}(\boldQ) < \frakf_{n,r}(K)$
to  $\frakf_{n,r}(\boldQ) < \frakf_{n,r}(L),$ where 
 $\frakf_{n,r}(\boldQ)$ denotes the $\boldQ$-span of the fixed Hall 
basis.   

Endow $\frakf_{n,r}(K)$ and $\frakf_{n,r}(L)$ with the rational structures
defined by  $\calB(K)$  and  $\calB(L)$ respectively.
  Given a matrix $A \in GL_n(\boldZ),$
let 
maps $f^A_K \in \Aut(\frakf_{n,r}(K))$ and $f^A_L \in \Aut(\frakf_{n,r}(L))$
 be defined by $A$ and  
  $\calB_1(K)$  and  $\calB_1(L)$ respectively. Observe that 
$[f^A_K]_{\calB(K)} = [f^A_L]_{\calB(L)} \in M_N(\boldZ),$ where
$N = \dim \frakf_{n,r}.$

Let $E$ be a rational $f^A_K$-invariant  subspace of  $\frakf_{n,r}(K)$
spanned by vectors $\bfv_1, \ldots, \bfv_m$ in $\frakf_{n,r}(K);$ i.e.
coordinates of $\bfv_1, \ldots, \bfv_m$ with respect to $\calB(K)$
are in $\boldQ.$
Define the subspace $E^L$ of $\frakf_{n,r}(L)$  to be the 
$L$-span of the  vectors $\overline{i}(\bfv_1), \ldots, 
\overline{i}(\bfv_m)$ in  $\frakf_{n,r}(L).$  Clearly 
 $E^L$ is rational and  $f^A_L$-invariant  subspace of  $\frakf_{n,r}(K).$
\end{definition}

\begin{remark}\label{f^2}  
Observe that $\fraki$ satisfies the Auslander-Scheuneman
conditions for a semisimple automorphism $f$ of a free nilpotent Lie algebra,
then it satisfies the conditions for $f^2.$  Therefore,
when seeking ideals of a free nilpotent Lie algebra satisfying the
four conditions for an automorphism $f,$ 
by moving to $f^2$ if necessary, we may assume
that the eigenvalues of the automorphism have product 1, and
that all the real eigenvalues are positive.
\end{remark}

The next example clarifies some of our definitions and notation.  
\begin{example}\label{free-3,2} 
Let $\frakf_{3,2}(\boldR) = V_1(\boldR) \oplus V_2(\boldR)$ be
the free two-step nilpotent Lie algebra on three generators $\bfx_1,
\bfx_2,$ and $\bfx_3.$  These three generators span the subspace
$V_1(\boldR).$ The Hall words of length two  are  $\bfx_1^\prime = [\bfx_3,
\bfx_2], \bfx_2^\prime = [\bfx_3,\bfx_1],$ and $\bfx_3^\prime =
[\bfx_2,\bfx_1];$ they span $V_2(\boldR).$  The union $\calB$ of  $\calB_1 =
\{ \bfx_1, \bfx_2, \bfx_3\}$ and  $\calB_2 = \{ \bfx_1^\prime,
\bfx_2^\prime, \bfx_3^\prime\}$ is the Hall basis determined by
$\bfx_1, \bfx_2, \bfx_3.$

Now let $A = A_1$ be a $3 \times 3$ matrix in $SL_3(\boldZ)$ that has
eigenvalues $\alpha_1, \alpha_2, \alpha_3,$ none of which
 has  modulus one.  The matrix $A$  and the basis
$\calB_1$ define the linear map $f_1 : V_1 \to V_1.$ The linear map
$f_1$ induces an automorphism  ${f^A}$ of  $\frakf_{3,2}(\boldR).$  Let  
$A_2$ denote the matrix representing the restriction of 
$f^A$ to ${V_2(\boldR)}$ with respect to
the basis ${\calB_2}.$  

The matrix $A_1$ has characteristic polynomial
\[ p_1(x) = (x - \alpha_1)(x - \alpha_2)(x-\alpha_3).   \] 
A short calculation shows that $A_2$ is similar to $A_1^{-1}$ and 
has characteristic polynomial  
\[
  p_2(x) =  (x - \alpha_2 \alpha_3)(x - \alpha_1
\alpha_3)(x-\alpha_1 \alpha_2) =  (x - \alpha_1^{-1})(x - \alpha_2^{-1})
(x-\alpha_3^{-1}) . \]
Neither $A_1$ or $A_2$ has any eigenvalues of modulus one, so ${f^A}$
is an Anosov automorphism of $\frakf_{3,2}(\boldR).$    
\end{example}

\subsection{Polynomials and algebraic
numbers}\label{polynomials-numbers}

 We will call a monic polynomial \textit{Anosov} if has integer
coefficients,  it has constant term 
$\pm 1$,  and it has no roots with modulus one.
The roots of an Anosov polynomial are algebraic units.

We can identify each monic polynomial $p$ in $\boldZ[x]$ of degree $n$
with an automorphism of the free nilpotent Lie algebra  $\frakf_{n,r}(\boldR)
= \oplus_{i = 1}^r V_i(\boldR)$ with generating set $\calB_1 = \{ \bfx_1, \ldots,
\bfx_n\}.$  Suppose $p = q_1 q_2 \cdots q_s$ is a factorization of
$p$ into  irreducibles.  Let $A_{i}$ be the companion matrix for
$q_i,$ for $i=1, \ldots, s,$ 
 and define  the  matrix $A_p$ to be block diagonal with the
matrices  $A_{1}, \ldots, A_{s}$ down the diagonal.  As 
already described,  the matrix $A_p$  and the  basis $\calB_1$ 
together define an automorphism of the free
$r$-step nilpotent Lie algebra  on $n$ generators. 

If $E$ is a nontrivial rational invariant subspace for an   Anosov
automorphism $f$  of an Anosov Lie algebra,
we will let  $p_E$ denote the characteristic polynomial for the
restriction of $f$ to $E.$ If $p$ and $q$ are polynomials in
$\boldZ[x],$ we define the polynomial $p \wedge q$ to be the
characteristic polynomial of the matrix $A_p \wedge A_q.$  

Next we illustrate how an Anosov polynomial determines a class of
Anosov  automorphisms. 
\begin{example}\label{qr} Let $p$ be an Anosov polynomial of degree
$n$ that is  a product of two irreducible  factors $r_1$  and $r_2$ of
degrees $d_1$ and $d_2$ respectively.  The companion matrices $B_1$
and $B_2$ to  the polynomials $r_1$ and $r_2$ are in $GL_{d_1}(\boldZ)$
and  $GL_{d_2}(\boldZ)$ respectively.  Putting these matrices together
in a block diagonal matrix gives a matrix 
\[ A_p = A_1 = \begin{bmatrix} B_1 & 0 \\ 0 & B_2 \end{bmatrix}\] in
$GL_{n}(\boldZ)$ with characteristic polynomial $p.$ 

Let $\frakf_{n,2}(\boldR) = V_1(\boldR) 
\oplus V_2(\boldR)$ be the real free two-step  nilpotent
Lie algebra on $n$ generators with generating set $\calB_1.$ The
matrix $A_1$  and the basis $\calB_1$ of $V_1(\boldR)$ define a linear map
$f_1: V_1(\boldR) \to V_1(\boldR)$ which induces an automorphism $f^A$ of
$\frakf_{n,2}(\boldR).$ Let $f_2 =  f^A|_{V_2(\boldR)}.$  The map $f_2$ may be
represented by a matrix $A_2$ that  is block diagonal with matrices
$B_1 \wedge B_1, B_1 \wedge B_2$ and $B_2 \wedge B_2$ along the
diagonal.   Let $\alpha_1, \ldots, \alpha_{d_1}$ denote the roots of
$r_1$ and let
 $\beta_1, \ldots, \beta_{d_2}$ denote the roots of $r_2.$ 
It can be shown that the matrix $A_2$ has characteristic
polynomial  
\[ p_2 = (r_1 \wedge r_1)(r_1 \wedge r_2)(r_2 \wedge r_2),\] where
\begin{align*} 
(r_1 \wedge r_1)(x) &= 
\prod_{1 \le i < j \le d_1} (x -\alpha_i \alpha_j), \\ 
(r_1 \wedge r_2)(x) &= 
\prod_{\substack{1 \le i \le d_1 \\ 1 \le  j \le d_2}} (x - \alpha_i \beta_j), \quad \text{and} \\
 (r_2 \wedge r_2)(x) &= \prod_{1 \le i < j \le d_2} (x - \beta_i \beta_j).
\end{align*}
As long as none of the roots of $p_2$ have modulus one, the map 
$f^A$ is Anosov.   
\end{example}

Later we will need to know when polynomials in $\boldZ[x]$  
have roots of modulus one and will use the following observation.  
\begin{remark}\label{roots of modulus one}   
Suppose that an irreducible polynomial  $p$ in $\boldZ[x]$ of degree $n$
has a root $\alpha$ with modulus one.
Then the complex  conjugate $\bar \alpha = \alpha^{-1}$ is also a
root of $p,$ so  $p$ is self-reciprocal and $n$ is even. If  $q$  is the
minimal polynomial of  $\alpha + \alpha^{-1},$ then $p(x) = x^{n/2} q(x +
1/x).$   Hence, the Galois group for $p$ is  a wreath product of  $C_2$
and the Galois group for the  polynomial $q.$ 
\end{remark}

\section{Polynomials associated to automorphisms}\label{polynomials}

\subsection{Properties of the $r$-tuple of characteristic polynomials}
\label{properties} Now we present some  properties of the tuple of
polynomials associated to an  automorphism of a free nilpotent Lie
algebra. 
\begin{proposition}\label{p-properties}  Let $f^A$  be a  semisimple
automorphism of the free nilpotent Lie algebra  $\frakf_{n,r}(\boldR)
= \oplus_{i=1}^r V_i(\boldR)$  defined by a matrix $A$ in
$GL_n(\boldZ)$ and the Hall basis $\calB$  defined by  generating set
$\calB_1 = \{ \bfx_i\}_{i=1}^{n}.$   Let $(p_1,  p_2, \ldots, p_r)$
be the $r$-tuple of polynomials associated to $f,$ let $\alpha_1,
\ldots, \alpha_n$ denote the roots of $p_1,$ and let $K$ denote the
splitting field for $p_1.$  Let $\calC_1 = \{\bfz_i\}_{i=1}^n$  be a
 $f^A_K$-eigenvector basis of $V_1(K) <  \frakf_{n,r}(K)$ compatible with the
rational structure defined by $\calB$   and let $\calC= \cup_{i=1}^r
\calC_i$ be the Hall basis of  $\frakf_{n,r}(K)$ associated to
$\calC_1.$
\begin{enumerate}
\item{Each standard Hall monomial of degree $i$ on $\bfz_1, \ldots,
\bfz_n$ in the set   $\calC_i$ is an eigenvector for 
$ f_{K}^A|_{V_i(K)}$ whose
eigenvalue is  the corresponding Hall word in   $\alpha_1,
\ldots, \alpha_n.$ }\label{eval-tilde-pi}
\item{For $i=1, \ldots, r,$ let $ p_i = r_{i,1} \cdots r_{i,d_i}$ be
a factorization of $ p_i$  into $d_i$ irreducible monic polynomials in
$\boldZ[x]$, and let  $V_i(\boldR) = \oplus_{j=1}^{e_i}E_{i,j}$ be a
decomposition of $V_i(\boldR)$ into $e_i$  minimal nontrivial rational 
$f^A$-invariant subspaces.   For all $i =1, \ldots, r,$
 $d_i = e_i$  and the
  map  that sends $E_{ij}$ to the  characteristic
polynomial of  $f|_{E_{ij}}$ is a one-to-one correspondence between
the set of rational subspaces $\{E_{ij}\}_{j = 1}^{d_i}$ and the set of
factors $\{ r_{i,j} \, : \, j =1, \ldots,
e_i \}$ of $p_i.$}\label{rat-subsp}
\end{enumerate}
\end{proposition}  

It follows from Part \eqref{eval-tilde-pi} of the proposition
 that  if the matrix $A$ is diagonalizable over $\boldC,$
then the automorphism $f^A$ is semisimple.  In particular, if the polynomial 
$p_1$ is separable over $\boldQ,$ then $f^A$ is semisimple.  

\begin{remark}\label{as cond 2} As a consequence of the third part of
the proposition,  if $f$ is unimodular,  the
characteristic polynomial for the restriction of $f$ to a rational invariant
subspace  $E$ has a
unit constant term, hence  the restriction of $f$ to
any rational invariant subspace $E$ is unimodular. 
Therefore, the second of the four
Auslander-Scheuneman conditions is automatic.
\end{remark}

 It is well known that there are no Anosov Lie
algebras of type $(n_1, \ldots, n_r),$ where $n_1 = 2$ and
 $r > 1.$  This follows from Part \ref{eval-tilde-pi}
 of the proposition.  Henceforth we shall only consider
nilpotent Lie algebras where $n_1 \ge 3$ and $r \ge 2.$

\begin{proof}[Proof of Proposition \ref{p-properties}.] 
The first part is elementary.  
 Hall words of degree $i$ in $\bfz_i$ span $V_i(K),$ 
for $i = 1, \ldots,
r.$ Because  $f_{K}^A$ is an automorphism of $\frakf_{n,r}(K),$  
a Hall word in $\bfz_1, \ldots,
\bfz_n$ is an eigenvector for $f_{K}^A$ whose 
eigenvalue is the same Hall word in $\alpha_1, \ldots, \alpha_n.$

The last part of the proposition follows from the existence of 
the elementary divisors rational canonical forms for matrices. 
The fact that the matrix is semisimple implies that the elementary
divisors are irreducible.   
\end{proof}

\subsection{The polynomials $p_2$ and
$p_3$}\label{p2p3} 

Let $A$ be a semisimple matrix in  $GL_n(\boldZ),$ and let $K$ be the
splitting field for the characteristic polynomial $p_1$ of $A.$ 
Let 
 $f^A$ be the semisimple automorphism of 
 $\frakf_{n,r}(\boldR) = \oplus_{i = 1}^r V_i(\boldR),$ where $n \ge 3,$
induced by  $A$ and a basis
$\calB_1$ for $V_1(\boldR).$   
 Let $\calC_1 = \{\bfz_i\}_{i=1}^n$ be an 
eigenvector basis for $ V_1(K) < \frakf_{n,r}(K)$
compatible with the rational structure determined by 
$\calB_1,$ where $f^A(\bfz_j) = \alpha_j \bfz_j,$ for $j = 1, \ldots, n;$
and let $\calC = \cup_{i=1}^r
\calC_i$ be the Hall basis of $\frakf_{n,r}(K)$ determined by $\calC_1.$

In Example \ref{2,3-Hall words} we described Hall words of length two
and three.  By Proposition \ref{p-properties}, the eigenvalue for 
an element $[\bfz_j,\bfz_i]$ of $\calC_2$ is $\alpha_i\alpha_j,$
so 
\begin{equation}\label{p2-def}  
p_2(x) = \prod_{1 \le j < i \le n} (x - \alpha_i \alpha_j).\end{equation} 
Let $\calC_3^\prime$ and $\calC_3^{\prime
\prime}$ be  as defined in Example \ref{2,3-Hall words}.
  By  Proposition \ref{p-properties}, an element
$[[\bfz_j,\bfz_i],\bfz_j]$
of $\calC_3^\prime$ is an eigenvector for
$f_{K}^A$ with eigenvalue  $\alpha_i \alpha_j^2,$ and an element
$[[\bfz_i,\bfz_j],\bfz_k]$ of $\calC_3^{\prime \prime}$ is an
eigenvector for $f_{K}^A$ with eigenvalue  $\alpha_i \alpha_j \alpha_k.$
Define the polynomials  $q_1$ and $q_2$ by 
\begin{equation}\label{q1,q2} q_1(x) = \prod_{1 \le i , j \le n, i \ne j}(x -
\alpha_i \alpha_j^2), \qquad \text{and} \qquad q_2(x) = \prod_{1 \le
i< j < k \le n}(x - \alpha_i \alpha_j \alpha_k).
\end{equation}
Because they are invariant under the action of $G,$ $q_1$ and $q_2$
have integral coefficients.  
The polynomial  $ p_3 = q_1
q_2^2$ is the characteristic polynomial for the restriction of the
automorphism $f^A_L$ 
to $V_3(L)$ for any extension $L$ of $\boldQ.$

\subsection{Anosov polynomials and their roots} In this section, we
discuss Anosov polynomials and their properties.   
\begin{proposition}\label{anosov poly}
 Let $p_1$ be an Anosov polynomial in
$\boldZ[x]$ of degree $n \ge 3.$   Let $(p_1, \ldots,
p_r)$ be the associated  $r$-tuple of polynomials. 
\begin{enumerate}
\item{
 If  $ p_1$ has
constant term one, then its reciprocal polynomial $( p_1)_R$   is a
factor of $ p_{n-1}$ and $( p_2)_R$  is a factor of $ p_{n-2}$.   If
the constant term of $ p_1$ is  $-1,$  then  $( p_1)_R(-x)$   is a
factor of $ p_{n-1}$ and $( p_2)_R(-x)$ is a factor of $ p_{n-2}$. 
}\label{recip}
\item{The polynomial 
 $p_i$ has at least one root of modulus greater than or equal to
than one and at least one root of modulus less than one
 for all $i \ge 2.$ }\label{root}
\item{If the roots $\alpha_1, \ldots, \alpha_n$ 
of $p_1$ are viewed as indeterminates, then the constant
term of $p_i$ is $(\alpha_1 \cdots \alpha_n)^{D(i)},$ with the
exponent $D(i)$  given by
\[ D(i) =  \frac{1}{ni}\sum_{d|i} \mu(d)
n^{i/d},\] where $\mu$ is the M\"obius function.}\label{dimVi}
\end{enumerate}
\end{proposition}

\begin{proof} Let $\alpha_1, \ldots, \alpha_n$ denote the roots of $p_1.$
Suppose  that the constant term of $ p_1$ is
$(-1)^{n},$ so $\alpha_1 \cdots \alpha_{n} = 1$. The reciprocal
of  $\alpha_j,$ for  $j  =1, \ldots, n,$  is  $
\alpha_{1} \cdots \hat \alpha_{j} \cdots  \cdots\alpha_{n},$  
 the Hall word $ \alpha_{n}
\cdots \hat \alpha_{j} \cdots  \cdots\alpha_{1}$ of length $n-1.$  By
Proposition \ref{p-properties}, Part \eqref{eval-tilde-pi}, this number is a
root of $p_{n-1}.$ Thus $(p_1)_R$ is a factor of $ p_{n-1}.$
Similarly,   the reciprocal of the root $\alpha_{j_1} \alpha_{j_2}$ of
$ {(p_2)_R},$, where $1 \le j_1 < j_2 \le n$, is  $\alpha_1 \cdots
\hat \alpha_{j_1} \cdots \hat \alpha_{j_2} \cdots\alpha_n,$  which is
a permutation of a Hall word on $n-2$ letters, and therefore is  a
root of $ p_{n-2}.$  Thus $(p_2)_R$ is a factor of $ p_{n-2}.$ If the
constant term of $p_1$ is $(-1)^{n+1},$  then $\alpha_1 \cdots
\alpha_n = -1$ and the same argument shows that whenever $\alpha$ is a
root of $ p_1,$ $-\alpha^{-1}$ is a root of $ p_{n-1},$ and when
$\alpha$ is a root of $ p_2,$ then $-\alpha^{-1}$ is a root of $
p_{n-2}.$ 

 Now we prove Part \eqref{root}. Fix   $i \ge 2.$ By
hypothesis, $\alpha_1 \cdots \alpha_n = \pm 1$ and not all of the
roots have modulus one.  Therefore, some root, say  $\alpha_1,$ has
modulus less than one.   Then  the modulus of the  product of the
$n-1$ other roots is $1/|\alpha_1| > 1,$ so   the modulus of  at
least one of the other roots, say $\alpha_2,$ must satisfy $|\alpha_2|
\ge  1/ {\sqrt[n-1]{|\alpha_1|}}$.  The root $\alpha_1^{i-1}
\alpha_2$ of $ p_{i}$ has modulus  equal to  $|\alpha_1|^{i-1 -
\frac{1}{n-1}}.$ The exponent $i-1 - \frac{1}{n-1}$ is positive, so
$\alpha_1^{i-1} \alpha_2$ has modulus less than one. 
By the same reasoning, there exists a root of $p_i$ with modulus
greater than one.

The dimension of $V_i(\boldR)$ is a Dedekind number   $\frac{1}{i}\sum_{d|i}
\mu(d) n^{i/d},$  where $\mu$ is the M\"obius function (Corollary
4.14, \cite{reutenauer}).    
Each of  $\alpha_1, \ldots, \alpha_n$  must occur the same number of times 
in the constant term of $p_i$ by Lemma 1 of \cite{auslander-scheuneman}. 
Therefore, the constant term of  $p_i,$ for
$1 \le i \le r,$ is $\alpha_1 \cdots \alpha_n$ to the power
 $1/n \cdot \dim(V_i(\boldR)),$ as claimed.
\end{proof}

The next lemma helps identify roots of modulus one 
for automorphisms whose first polynomial $p_1$ has Galois group of 
odd order.

\begin{lemma}\label{odd-good} Suppose that the characteristic
polynomial $p_1$ of a semisimple hyperbolic matrix $A$ in $GL_n(\boldZ),$
where $n \ge 3,$ has
Galois group  of odd order.  Let 
$f^A: \frakf_{n,r}(\boldR) \to \frakf_{n,r}(\boldR)$
be  the automorphism of the  free $r$-step nilpotent Lie  algebra on
$n$ generators induced by $A$. If $\lambda$ is an eigenvalue of
$f^A$ with modulus
one, then $\lambda = 1$ or $\lambda = -1.$   
\end{lemma}

\begin{proof} Let $p_1$ and $f^A$ be as in the statement of the lemma. Let 
$(p_1, \ldots, p_r)$ be the  $r$-tuple of polynomials associated
to $f^A.$

If a monic irreducible nonlinear  polynomial in $\boldZ[x]$ has a root
of modulus one,  by Remark \ref{roots of modulus one}, its Galois
group $G$ has even order.  

By Proposition \ref{p-properties}, if $\alpha$ is an eigenvalue of 
$f^A,$ it is a root of an irreducible factor $q$ of $p_i$ for 
some $i = 1, \ldots, r.$   
Since the splitting field for $q$ is a subfield of the splitting
field for $p_1,$ the  Galois group $H$ for $q$ is  the quotient of $G$ by a
normal subgroup; hence if $q$ is nonlinear,
 $H$ has odd order.  Then, either $q$ is linear and 
 $\alpha = \pm 1,$ or 
 $q$ is nonlinear and $|\alpha| \ne 1.$
\end{proof}

\subsection{The full rank condition}

Suppose that $p_1$ is an Anosov polynomial in $\boldZ[x]$  with
roots  $\alpha_1, \ldots, \alpha_n.$  
We will want to know 
 when the equation   
\begin{equation}\label{system} \alpha_1^{d_1} \alpha_2^{d_2} \cdots
\alpha_n^{d_n} =  1\end{equation}
has  integer solutions $d_1, \ldots, d_n.$  Note that if 
 $p_1$ has constant term $(-1)^n,$  then $\alpha_1 \cdots \alpha_n  =1,$ 
and  $d_1 = \cdots = d_n = d$ is a solution for any integer $d.$

\begin{definition}\label{full rank-def}
 Let $\Lambda= \{\alpha_1, \ldots, \alpha_n\}$
 be the set of roots of a polynomial $p$ in $\boldZ[x]$ 
with constant term $(-1)^n$
and degree $n \ge 2.$  The set $\Lambda$ is said to be of {\em full rank}
if the only integral 
solutions to Equation \eqref{system} are of form $d_1 = d_2 = \cdots = d_n.$
\end{definition}

 The next proposition describes how multiplicative relationships among the
roots of some polynomials in $\boldZ[x]$ depend on their Galois groups.  
\begin{proposition}\label{full rank}  Suppose that $\alpha_1, \ldots
,\alpha_n$ are roots of a degree $n$ irreducible monic polynomial $p$ in
$\boldZ[x]$ with constant term $(-1)^n,$ and suppose that none of
  $\alpha_1, \ldots
,\alpha_n$ are roots of unity.  Let $G$ denote the Galois group for
$p.$

 The set $\{\alpha_1, \ldots,\alpha_n\}$ is of full rank 
in the following situations. 
\begin{enumerate}
\item{When the permutation representation of 
$G$ on $\boldQ^n$ is the sum of the principal representation and a
representation that is irreducible over $\boldQ.$  
}\label{irr-rep}
\item{When the action of $G$  on the set of roots of $p$
is doubly transitive.}\label{2t->fr}
\item{When $p$ is Anosov, 
and  precisely one  of its roots $\alpha_1$
has modulus greater than one.}\label{bh-gen}
\end{enumerate}
\end{proposition}
  
An algebraic number $\alpha_1$ as in Part \eqref{bh-gen} of the
proposition is  a \textit{P.-V.\ number}.   
Properties  of P.-V.\ numbers were
first investigated by  Pisot and Vijayaraghavan.  (See \cite{meyer-72}
and \cite{bertin-92} for background on P.-V.\  numbers.) 
The proof of 
 Part \eqref{bh-gen} of the proposition is due to  
Bell and Hare (\cite{bell-hare-05}); we repeat it here for the
sake of completeness.

The action of the  the Galois group
$G$ of $p_1 \in \boldZ[x]$ on the set $\{\alpha_1, \ldots, \alpha_n\}$
of enumerated roots 
of $p_1$ gives an identification of $G$ with a subgroup of $S_n,$ and 
we can  define a  permutation representation $\rho$ of $G$ on $\boldQ^n,$
with 
\begin{equation}\label{rho-def} 
\rho(g) (\beta_1, \ldots, \beta_n) = (\beta_{g(1)}, 
\ldots, \beta_{g(n)})\end{equation}
for $g \in G$ and $(\beta_1, \ldots, \beta_n) \in \boldQ^n.$

\begin{proof}  Fix $\alpha_1, \ldots, \alpha_n$ as in the statement of the 
theorem.
If $(d_1, \ldots, d_n) \in \boldQ^n$ is a solution to 
Equation \eqref{system}, and $\sigma$ is in $G,$ then 
\begin{equation*} \sigma(\alpha_1)^{d_1} \sigma(\alpha_2)^{d_2} 
\cdots \sigma(\alpha_n)^{d_n} =  1,\end{equation*}
which may  be alternately expressed as
\begin{equation}\label{system-2}
\alpha_1^{d_{\sigma^{-1}(1)}} \alpha_2^{d_{\sigma^{-1}(2)}}  
\cdots \alpha_n^{d_{\sigma^{-1}(n)}}  =  1.\end{equation}
Therefore, the set of integral solutions to Equation 
\eqref{system}  is invariant for 
the permutation representation $\rho:$  For all $\sigma$ in $G$,  
  $(d_1, \ldots, d_n) \in \boldQ^n$ is a solution to Equation 
\eqref{system} if and only if $(d_{\sigma^{-1}(1)}, \ldots, d_{\sigma^{-1}(n)})$ is  a solution
to Equation \eqref{system}.  
It is easy to see that the set $S$ of solutions to Equation \eqref{system}
in $\boldZ^n$ is closed 
under addition and subtraction. Therefore, if  
$(d_1, \ldots, d_n)$ is an integral solution to  Equation \eqref{system}, 
then any  vector in  
\[\myspan_{\boldZ} \{ \rho(\sigma)(d_1, \ldots, d_n) \, : \, \sigma \in G \}\]
is also a   solution to  Equation \eqref{system}.

Suppose that $\rho$
decomposes as the sum of the trivial representation on $\boldQ(1,1, \ldots, 1)$
and an irreducible representation on  $W = (1,1, \ldots, 1)^\perp.$  
We will show by contradiction that the set $\{\alpha_1, \ldots, \alpha_n\}$
has full rank.  
Suppose that $(d_1, \ldots, d_n)$ is an $n$-tuple of integers so that
Equation \eqref{system} holds and that $(d_1, \ldots, d_n)$
is not a scalar multiple of  $(1,1, \ldots, 1).$  After
subtracting the appropriate multiple of  $(1,1, \ldots, 1),$ 
we may assume that the solution
  $(d_1, \ldots, d_n)$ is a nontrivial vector in $W.$   The 
representation of $G$ on $W$ is irreducible over $\boldQ,$
 and  $(d_1, \ldots, d_n)$
is a nontrivial element of $W,$ so the invariant subspace
 $\myspan_{\boldQ} 
\{ \rho(\sigma)(d_1, \ldots, d_n) \, : \, \sigma \in G \}$ is 
 all of $W.$
Then $\myspan_{\boldQ} S = \boldQ^n,$ 
implying that 
$(1, 0, \ldots, 0)$ is a $\boldQ$-linear combination of 
solutions $(d_1, \ldots, d_n)$ to Equation \eqref{system}.  
But then there exists an integer $N$ such that $\alpha_1^N = 1,$
a contradiction.  Hence, every solution to  Equation  \eqref{system}
is a scalar multiple of $(1,1, \ldots, 1).$

If 
the action of $G$ is  two-transitive, then the permutation representation of $G$ on $\boldC^n$
is the sum of the trivial representation on $\boldC(1,1, \ldots, 1)$
and a representation on the orthogonal complement that is 
irreducible over $\boldC$  (\cite{serre-77}, Exercise 2.6).  Then 
the  the permutation representation of $G$ on $\boldQ^n$
is the sum of the trivial representation on $\boldQ(1,1, \ldots, 1)$
and a representation on the orthogonal complement that is 
irreducible over $\boldQ,$ and
 the set of roots is of full rank, by Part \eqref{irr-rep}.
Therefore, if $G$ is two-transitive, then
 the set of roots is of full rank.

Now assume that $p$ is irreducible, and without loss of generality
that $\alpha_1 > 0.$  
Suppose that the  roots of $p$ satisfy the condition that
\[ \alpha_1  > 1  > |\alpha_2| \ge \cdots \ge |\alpha_n|, \] 
and  that Equation \eqref{system} holds for $d_1, d_2, \ldots, d_n.$
 Let $m$ be the index so that  
$d_m$ achieves the minimum of the set 
$\{ d_i \, : \, i=1, \ldots, n\}.$ Since 
$(d_m, \ldots, d_m)$  is a solution to Equation \eqref{system},
the $n$-tuple 
\[ (e_1, \ldots, e_n) = (d_1 - d_m, d_2 - d_m, \ldots, d_n -d_m)\] 
is a solution
to Equation \eqref{system} with  $e_m = 0$ and $e_i \ge 0$ for all 
$i = 1, \ldots, m.$

Because $p$ is irreducible, there exists a 
  permutation  $\sigma$ in $G$ such
that  $\sigma(\alpha_1) = \alpha_m,$ or if we identify 
 $G$ with a subgroup of $S_n$ in the natural way, $\sigma(1) = m.$  We
then have
\[  \rho(\sigma)(e_1, \ldots, e_n)  =   
(e_{\sigma(1)}, e_{\sigma(2)},
\ldots, e_{\sigma(n))}) =  (0, e_{\sigma(2)}, \ldots,
e_{\sigma(n))})\] 
is also a solution to the equation, so
\[ \alpha_2^{e_{\sigma(2)}}
\alpha_3^{e_{\sigma(3)}} \cdots 
\alpha_n^{e_{\sigma(n)}} =  1. \]  But 
 $|\alpha_i| < 1$ for $i = 2, \ldots, n,$ and  
 all the exponents are nonnegative, hence all the
exponents $e_i$ must be zero.  Then $d_1 = d_2 = \cdots = d_n$ as desired,
and Part \eqref{bh-gen} holds.
\end{proof}

The next lemma shows that
when the set of roots of an Anosov polynomial $p_1$ is of full rank,
there are strong restrictions on the Galois groups for irreducible 
factors of polynomials $p_2, p_3, \ldots$ associated to $p_1.$
\begin{lemma}\label{stabilizer}
Let $p_1$ be an Anosov polynomial of degree $n$ with 
 constant term $(-1)^n$.  Suppose that  the set of roots 
$\{\alpha_1, \ldots, \alpha_n\}$ of $p_1$ has full rank.
  Let  $G$ denote the Galois group of $p_1,$ and  let $(p_1, \ldots, p_r)$
be the $r$-tuple of polynomials associated to $p_1$ for some $r>1.$

Fix  $p_i$ for some $i = 2, \ldots, r,$ and
suppose that $q$ is an irreducible nonlinear factor  of $p_i$ over $\boldZ$
 with root 
$\beta = \alpha_1^{d_1} \cdots \alpha_s^{d_s}$ for $s \le n-1.$
The Galois group $G$ acts on the  splitting 
field $\boldQ(p_1)$ for $p_1,$ and  the splitting
field $\boldQ(q)$ for $q$ is a subfield of $\boldQ(p_1).$
 Let $H < G$ be the stabilizer of 
$\boldQ(q).$  
 Any element $\sigma$ in $H$  has the properties 
\begin{enumerate}
\item{$\sigma$ permutes the set  $\{ \alpha_{s+1}, \ldots, \alpha_n\}.$}\label{permute-set}
\item{For $j, k = 1, \ldots, s,$
 $\sigma(\alpha_j) = \alpha_k$ only if $d_j = d_k.$  That is, 
$\sigma$ permutes the sets of roots having the same exponent in 
the expression for $\beta$ in terms of $\alpha_1, \ldots, \alpha_n.$}
\end{enumerate}
Thus, $H$  is  isomorphic to a subgroup of the direct product  
\[S_{k_1} \times S_{k_2} \cdots \times S_{k_{m-1}} \times S_{n - s}, \qquad 
(k_1 + \cdots + k_{m-1} = s) \]
of $m \ge 2$ symmetric groups. 
\end{lemma}

\begin{proof}   Let $\beta = \alpha_1^{d_1} \cdots \alpha_s^{d_s}$ be as
in the statement of the theorem, and let $\sigma$ be in the stabilizer $H$ of
$\boldQ(q).$  Then $\sigma(\beta) = \beta$ implies that
\[  \alpha_1^{d_{\sigma^{-1}(1)}} \alpha_2^{d_{\sigma^{-1}(2)}} \cdots 
\alpha_s^{d_{\sigma^{-1}(s)}} \alpha_n^0 =  \alpha_1^{d_1} \cdots \alpha_s^{d_s},\]
so then 
\[  \alpha_1^{d_{\sigma^{-1}(1)} - d_1} \alpha_2^{d_{\sigma^{-1}(2)} - d_2} \cdots 
\alpha_s^{d_{\sigma^{-1}(s)} - d_s} \alpha_n^0 =   1.\]
By definition of full rank,  
\[ d_{\sigma^{-1}(1)} - d_1 = d_{\sigma^{-1}(2)} - d_2 = \cdots =
d_{\sigma^{-1}(s)} - d_s = 0.\]
Therefore, $d_{\sigma^{-1}(i)} = d_i$ for all $i = 1, \ldots, s;$ in other
words, if $\sigma(i) = j,$ then $d_i = d_j.$  Hence
$\sigma$ permutes each set of $\alpha_i$'s for which the exponents $d_i$
 agree, including the nonempty
 set $\{ \alpha_{s+1}, \ldots, \alpha_n\}$ where $d_i = 0.$  
\end{proof}

The following lemma describes how the polynomials in the 
$r$-tuple of polynomials  $(p_1, p_2, \ldots, p_r)$
can factor, in terms of the properties of $p_1.$
\begin{lemma}\label{how it can factor}   Suppose  $p_1$
is a monic polynomial in $\boldZ[x]$ of degree $n \ge 3$ with constant
term $(-1)^n.$ 
Let $G$ denote the Galois group for $p_1.$ Let $(p_1, p_2, \ldots, p_r)$
be an $r$-tuple of polynomials associated to $p_1.$  
 Let $q_1$ and $q_2$ be as defined
in  Equation \eqref{q1,q2}.
\begin{enumerate}
\item{
  Assume that $p_1$ is separable. 
If  $p_2$ or $q_1$ factors over $\boldZ$ as a power of an irreducible
polynomial $r,$ then the degree of $r$ is $n-1$ or more. 
If  $q_2$ factors over $\boldZ$ as a power of an
irreducible polynomial $r,$ then the degree of $r$ is $n-2$ or more.}\label{degree-r}
\item{Suppose that 
 the set of roots of $p_1$ is of full rank.  
\begin{enumerate}
\item{For $i =2, \ldots, r,$ 
the degree $k$ of  any factor of $p_i,$  satisfies 
$k = d (n/i)  $ for some positive integer $d \le D(i),$
where 
 $D(i)$ is as defined in Proposition \ref{anosov poly}. 
Therefore, if $\gcd(n,i)=r,$ then $n/r$ divides $k.$ }\label{factor-degree}
\item{
 If $q$ is a nonlinear irreducible factor  of $p_i$ over $\boldZ$
for some $i = 2, \ldots, r,$
 then the normal subgroup $N$ of $G$ of automorphisms 
 fixing $\boldQ(q)$  does not  act
transitively on the roots of $p_1.$   }\label{N-not-trans}
\item{If $p_2 = r^s$ or  $q_1 = r^s$ for an irreducible monic 
polynomial $r \in \boldZ[x],$ 
then $s = 1$   when $n \ge 3;$ and  when $n \ge 4,$
if  $q_2 = r^s$  for an irreducible $r,$ then $s=1.$}\label{fr-p2-irr}
\end{enumerate}}
\end{enumerate}
\end{lemma}

\begin{proof} 
Let  $\alpha_1,
\ldots, \alpha_n$  denote the roots of $p_1.$
 If a polynomial $r$ is
irreducible with $m$ distinct roots, then $r^s$ has $m$ distinct roots
each of multiplicity $s.$  Thus, to prove the 
 the first part, we simply count the number of roots of each polynomial
$p_2, q_1,$ and $q_2$ that are guaranteed to be distinct, and the
degree of $r$ is necessarily greater or equal to that number if 
$p_2, q_1,$ or $q_2$ is of form $r^s.$ For 
$p_2,$ roots of form $\alpha_1 \alpha_j, j = 2, \ldots, n$ are distinct; for 
$q_1,$ roots of form $\alpha_1 \alpha_j^2, j = 2, \ldots, n$ are distinct;
and  for 
$p_2,$ roots of form $\alpha_1 \alpha_2 \alpha_j, 
j = 3, \ldots, n$ are distinct.  This proves Part \eqref{degree-r}.

Now suppose that the set of roots of $p_1$ has full rank.  
Let $q$ be a degree $k$ 
factor of $p_i$ over $\boldZ[x]$ for some $i = 2, \ldots, n.$
The constant term  of $q$ is $\pm 1,$ and by the full rank property, of 
form $(\alpha_1 \alpha_2 \cdots \alpha_n)^d$
for some positive integer $d.$  The constant term of $p_i$ is 
 $(\alpha_1 \cdots \alpha_n)^{D(i)},$ by Proposition \ref{anosov poly},
so $d \le D(i).$
On the other hand, roots of $p_i$ are Hall words of degree $i$ in 
$\alpha_1, \alpha_2, \ldots, \alpha_n,$ so the constant term of $q$ is 
the product of $k$ $i$-letter words in  
 $\alpha_1, \alpha_2, \ldots, \alpha_n.$  Thus, $ki = nd,$
and the degree $k$ is an integral multiple of $n/i$ as desired.
This proves Part \eqref{factor-degree}.

Now suppose that $q$ is nonlinear and irreducible.   
Let $\beta$ be a root of $q.$  Using the identity 
 $\alpha_1 \cdots \alpha_n = 1,$ we may write $\beta$ in the form
 $\alpha_1^{d_1} \cdots \alpha_{n-1}^{d_n}$ where no $\alpha_n$ appears,
and    Lemma \ref{stabilizer} applies.  
Then by the lemma, the action of $N$ on $\{\alpha_1 \cdots \alpha_n\}$
is not transitive.

Finally, to show irreducibility of $p_2, q_1$ and $q_2,$ use
the full rank condition to show that the roots of $p_2$ and
 $q_1$ are distinct when $n \ge 3,$ and the roots of $q_2$ are 
distinct when $n \ge 4.$
\end{proof}

We obtain a corollary that describes the dimensions of the steps
of certain sorts of Anosov Lie algebras. 
\begin{corollary}\label{prime-dim}
Let $\frakn$ be an $r$-step nilpotent Lie algebra of type $(n_1, \ldots, n_r),$
where  $n_1$ is prime and 
$1 < r < n_1,$ that admits an Anosov automorphism $f.$  
Let $(p_1, p_2, \ldots, p_n)$ be the $r$-tuple of polynomials associated
to an automorphism $f$ of $\frakf_{n,r}(\boldR)$ that has $\overline{f}$ as a 
quotient. 
Suppose that  the polynomial $p_1$ is irreducible.  
Then $n_1$ divides $n_i$ for all $i =2,\ldots, r.$
\end{corollary}

\begin{proof}
Let $G$ denote the Galois group of the polynomial $p_1$ 
associated to the Anosov automorphism $f.$ 
Because $p_1$ is irreducible,  its roots are distinct, 
and $f$ is semisimple.   
The number $n_1$ is prime, hence
it divides the order of $G,$ and by Lagrange's Theorem, there
is a subgroup of $G$ isomorphic to $C_{n_1}.$ 
 The permutation representation 
of $G$ on $\boldQ^n$ is then the sum of the principal representation 
and a representation that is irreducible over $\boldQ.$  Hence the
set of roots of $p_1$ has full rank by Proposition \ref{full rank}.

Now let $q$ be the characteristic polynomial of a rational $f$-invariant 
subspace $E$ contained in $V_i(\boldR)$ for some $i < n.$    Because $n_1$ is
prime, and $i < n_1,$ the numbers $i$  and $n_1$ 
are coprime.  Therefore,  the dimension $k$  of
a rational invariant subspace for $f$ is an integral multiple of $n_1,$
by Part \eqref{factor-degree} of Lemma \ref{how it can factor}.
Therefore, the dimension $n_i$ of the $i$th step of $\frakn$ is a multiple of $n_1$
for all $i = 2, \ldots, r.$
\end{proof}

\subsection{The existence of Anosov polynomials with given Galois group}
 
In the next proposition, we summarize some results on the existence
of Anosov polynomials with certain properties. 

\begin{proposition}\label{existence} 
There exist irreducible Anosov polynomials in
$\boldZ[x]$ satisfying the following  conditions.
\begin{enumerate}
\item{For all $n \ge 2,$ for all $r = 1, \ldots n-1,$
there exists an irreducible Anosov polynomial
$p$ of degree  $n$   such that precisely $r$ of the roots have modulus larger
than one. 
 }\label{pisot}
\item{For all  $n \ge 2,$ there exists an irreducible Anosov
polynomial of degree $n$ with Galois group  $S_n.$ }\label{G=Sn}
\item{For all prime $n \ge 2,$ there exists an irreducible Anosov
polynomial of degree $n$ with Galois group $C_n.$}
\item{Suppose that the group $G$ acts transitively on some set
of cardinality 2, 3, 4 or 5 (hence is the Galois group of an irreducible 
polynomial in $\boldZ[x]$ of degree  $n \le 5$), and $G$ is not isomorphic
 the alternating group $A_5.$ 
 Then  there exists an Anosov polynomial of 
degree $n$ having Galois group $G.$}\label{low-degree}
\end{enumerate}
\end{proposition}

D.\ Fried showed how to use the geometric version of the 
Dirichlet Unit Theorem and
 the results from \cite{auslander-scheuneman} to construct 
Anosov automorphisms with given spectral properties (\cite{fried-81});
these methods can also be used to prove the first part of the proposition.
We provide an alternate proof that shows the actual polynomials defining
the automorphisms.  

\begin{proof}
The polynomial $p(x) = x^n + a_1x^{n-1} + \cdots + a_{n-1}x \pm 1$
has $r$ roots greater than one in modulus and $n-r$ roots less
than one in modulus when 
\[ |a_r| > 2 + |a_1| + \cdots + |a_{r-1}| + |a_{r+1}| + \cdots + |a_{n-1}|\]
(\cite{mayer}, see \cite{tajima}).
By letting $a_r=3$  and $a_i = 0$ for $i \ne r$ in $\{1, \ldots, n-1\},$
we get a polynomial of degree $n$ with precisely
 $r$ roots greater than one in modulus 
that is irreducible by Eisenstein's Criterion.  By  Remark
\ref{roots of modulus one},  if $n \ne 2r,$
the polynomial can not have any roots of
modulus one.  If $n=2r,$ then the polynomial is 
$(x^r)^2 + 3x^r \pm 1;$ one can check by hand that neither of these polynomials 
have roots of modulus one. 
 This proves
 Part \eqref{pisot}.

The irreducible  polynomial  $x^n - x -1$ has Galois  group $S_n$
(\cite{serre-92}).   As it is not self-reciprocal, by Remark
\ref{roots of modulus one},  it has no roots of modulus one. 

Now suppose that $K$ is a Galois
extension of $\boldQ$ with Galois group $C_n$ with $n \ge 3$ prime.  
It is well known  that such a field exists (See the Kronecker-Weber Theorem, in
\cite{jensen-ledet-yui}).
Let $\eta$ be a Dirichlet fundamental unit for $K,$ and let 
$p$ denote its minimal polynomial.  Since $\eta$ is a unit, the constant
term of $p$ is $\pm 1.$   Because  $\eta \not \in \boldQ,$
the degree $m$ of $\eta$ is greater that one.  
Because $m$ divides $n,$ and $n$ is prime,
 $m$ must equal $n.$ Thus, the minimal polynomial
$p$ for $\eta$ has degree $n$ and  splitting field $K.$   The
degree of $p$ is odd, so by Lemma \ref{odd-good},
$p$ has  no roots of modulus one. Thus we have shown that $p$ is
an Anosov polynomial.  
 
Now we prove Part \eqref{low-degree}.
It is simplest to list examples
of Anosov polynomials of each kind:  See Table \ref{low-degree-list}.
 By Remark
\ref{roots of modulus one}, the only polynomial in the
table that could possibly have roots of modulus one is the  self-reciprocal
  polynomial with Galois group $V_4.$  An easy calculation shows
that it does not have roots of modulus one.  

\renewcommand{\arraystretch}{2}
\begin{table}
\begin{tabular}{|c|c|l|} \hline 
 Degree & Galois group & Anosov polynomial
\\ \hline \hline 
2 & $C_2$ &  $p(x) = x^2 - x - 1$ \\ 
\hline \hline 
3 & $C_3$ &  $p(x) = x^3 - 3x - 1$  \\ 
\hline 
3 & $S_3$ &  $p(x) = x^3 - x - 1$ \\ 
\hline \hline 
4 & $C_4$ & $p(x) = x^4 + x^3 - 4x^2 - 4x + 1$\\ 
\hline 
4 & $V_4$ &  $p(x) = x^4 +  3x^2  + 1$\\ 
\hline 
4 & $D_8$ &  $p(x) = x^4 -x^3 -x^2 + x + 1$\\  
\hline 
4 & $A_4$ &   $p(x)= x^4 + 2x^3 + 3x^2  - 3x + 1$ \\ 
\hline 
4 & $S_4$ &   $p(x)= x^4 - x - 1$ \\
\hline \hline 
5 & $C_5$ & $p(x) = x^5 + x^4 - 4x^3 - 3x^2 + 3x + 1$  \\
\hline 
5 & $D_{10}$ & $p(x) = x^5 - x^3 -2x^2 - 2x -1$ \\ 
\hline 
5 & $F_{20}$ &  $p(x) = x^5 + x^4 + 2x^3 + 4x^2 + x + 1$ \qquad \strut \\
\hline 5 & $A_5$ & Q: What is an example with small coefficients? \qquad \strut  \\ 
\hline 5 & $S_5$ &  $p(x)= x^5 - x - 1$ \\
\hline
\end{tabular}

\bigskip

\caption{\label{low-degree-list} The inverse Galois problem for Anosov polynomials of low
degree}
\end{table}
\end{proof}  
Many of the examples in Table \ref{low-degree-list} were
taken from the appendix of \cite{malle-matzat}; the reader may find there
a great many more examples of Anosov polynomials of degree $n \ge 6$ 
with a  variety of Galois  groups. To our knowledge, it
is known whether the existence of a polynomial in $\boldZ[x]$
with  Galois group $G$ guarantees the existence of
a  polynomial in $\boldZ[x]$ with 
constant term $\pm 1$ and Galois group $G.$  In addition, we do not know of
an example of an Anosov polynomial with Galois group $A_5.$

\section{Actions of the Galois group}\label{action-section}

\subsection{Definitions of actions}  
 In this section, we associate the $\boldQ$-linear action of a finite group
to an automorphism of a free nilpotent Lie algebra that
preserves the rational structure defined by a Hall basis.

\begin{definition}\label{action} 
Let $A$ be a matrix in $GL_n(\boldZ),$ let $p_1$ be the
characteristic polynomial of $A$ and let $K$ be the splitting field for 
$p_1.$  Suppose  that $f$ is the automorphism of
$\frakf_{n,r}(K)$ determined by $A$
and a set  $\calB_1= \{ \bfx_i\}_{i=1}^n$ of generators  for
$\frakf_{n,r}(K).$  Let $\calB = \cup_{i=1}^r \calB_i$ 
be the Hall basis determined by $\calB_1.$
Write $\frakf_{n,r}(K) = \oplus_{i=1}^r V_i(K),$ and
 for $i = 1, \ldots, r,$ let $n_i = \dim_K V_i(K).$

Let $G$ denote the Galois group for the field $K.$ 
 Let
$\bfx_1^i, \ldots, \bfx_{n_i}^i$ denote the $n_i$ elements of the basis
$\calB_i$ of $V_i(K);$ they determine an identification  
 $V_i(K) \cong K^{n_i}.$ 
For $i=1, \ldots, r,$  the  $G$ action on  $K$
extends to a diagonal $G$ action on  $V_i(K) \cong K^{n_i}.$ 
In particular, if $\bfw = \sum_{j=1}^{n_i} \beta_j \bfx_j^i$ is
an  element  in $V_i(K),$ where $\beta_1, \ldots, \beta_{n_i} \in K,$
and $g \in G,$ then   $g \cdot \bfw$ is defined by
\[ g \cdot \bfw =  \sum_{j=1}^{n_i}  (g  \cdot \beta_j) \, 
\bfx_j^i \in  V_i(K). \] 
A $G$ action on the 
free nilpotent Lie algebra  $\frakf_{n,r}(K) = \oplus_{i=1}^r V_i(K)$ 
is defined  by
 extending each of the $G$ actions on $V_i(K),$ for $i = 1, \ldots, r.$
\end{definition}

Next we describe properties of the action.  
\begin{proposition}\label{action-properties} Let 
$A$ be a semisimple matrix in  $GL_n(\boldZ),$
let  $p_1$ be the characteristic polynomial of $A,$ and let
$K$ be the splitting field for $p_1$ over $\boldQ.$
Let $\calB_1$ be a generating set for the free nilpotent
Lie algebra $\frakf_{n,r}$ and let $\calB$ be the Hall 
basis that it determines.  
Let  $f$ be the automorphism of  $\frakf_{n,r}(K),$ induced by Hall basis
$\calB$ and the matrix $A.$  Let $G$ be the
Galois group for $K,$ and let $G \cdot
\frakf_{n,r}(K) \to \frakf_{n,r}(K)$ be the action defined in
Definition \ref{action}. Then
\begin{enumerate}
\item{The $G$ action is $\boldQ$-linear, preserving
$\frakf_{n,r}(\boldQ) < \frakf_{n,r}(K),$ and it preserves
 the decomposition $\frakf_{n,r}(K)  =
\oplus_{i=1}^r V_i(K)$ of $\frakf_{n,r}(K)$ into steps.}\label{preserve-Vi}
\item{The $G$ action on $\frakf_{n,r}(K)$ commutes with the Lie
bracket.  }\label{action-bracket}
\item{The function  $f: \frakf_{n,r}(K) \to \frakf_{n,r}(K)$ 
is $G$-equivariant.}\label{action-fK}
\item{The $G$ action permutes the eigenspaces of $f:$ 
an element $g \in G$ 
sends the $\alpha$ eigenspace for $f$ to the $g
\cdot \alpha$ eigenspace for $f.$}\label{action-eigenspace}
\end{enumerate}
\end{proposition}
  
Note also that because it fixes $\frakf_{n,r}(\boldR)$ stepwise, 
the $G$ action on $\frakf_{n,r}(K) = \oplus_{i=1}^r V_i(K)$
commutes with the  grading automorphism $B$ 
defined by $B (\bfw) = e^i \bfw$ for $\bfw \in V_i.$

\begin{proof} The first assertion follows from the definition of the
action.

The action of $G$ commutes with the Lie bracket because
structure constants for the Hall basis $\calB$
are rational. Let $\bfy_1, \ldots, \bfy_N$  be an enumeration of
the Hall basis $\calB,$ and denote the structure constants relative to 
$\calB$  by
$\alpha_{ij}^k.$  Consider the Lie bracket of two 
arbitrary vectors $\sum_{i=1}^N a_i \bfy_i$
and $\sum_{j=1}^N b_j \bfy_j$ in $\frakf_{n,r}(K):$
\begin{align*}  g \cdot \left[\sum_{i=1}^N a_i \bfy_i, \sum_{j=1}^N
b_j \bfy_j\right]  &=  g \cdot \sum_{k=1}^N \left(  \sum_{j=1}^N
\sum_{j=1}^N  \alpha_{ij}^k a_i b_j \right) \bfy_k \\  &= \sum_{k=1}^N
\left(  \sum_{j=1}^N  \sum_{j=1}^N  \alpha_{ij}^k (g \cdot a_i) (g
\cdot b_j) \right) \, \bfy_k \\  & = \left[  g \cdot \sum_{i=1}^N a_i
\bfy_i,  g \cdot \sum_{j=1}^N b_j \bfy_j \right]. 
\end{align*}

Now we show that  the $G$ action on $\frakf_{n,r}(K)$ commutes with
the  automorphism $f.$     We can 
write $\bfw$ in $\frakf_{n,r}(K)$  as a linear combination 
$\sum_{j=1}^{N} \beta_j \bfy_j$ of elements in the Hall basis 
$\calB$ of $\frakf_{n,r}(K).$  Take  $g \in G.$
Because $f$ is linear, 
\[ f (g \cdot \bfw) =  f \left(\sum_{j=1}^{N} (g \cdot \beta_j)
\bfy_j \right) =  \sum_{j=1}^{N} (g \cdot \beta_j) f (\bfy_j) .\]
An integral matrix  $(c_{ij})$ represents $f$ with respect to
the basis $\calB.$  
 Therefore, for each $j=1, \ldots, N,$ the vector
$f(\bfy_j)$ is a $\boldZ$-linear combination $ \sum_{i=1}^{N}
c_{ij} \bfy_i$  of $\bfy_1, \ldots, \bfy_N.$  All of the
integer entries of the matrix $(c_{ij})$  are fixed by the
$G$ action. Hence, $ f (g \cdot \bfw)$ equals
\begin{align*} 
\sum_{j=1}^{N} (g \cdot \beta_j) f (\bfy_j) &= 
\sum_{j=1}^{N} (g \cdot \beta_j)  \sum_{i=1}^{N} c_{ij}\, \bfy_j\\ 
&= g \cdot    \sum_{j=1}^{N} \sum_{i=1}^{N}  \beta_j
c_{ij} \, \bfy_j\\
&= g \cdot  f \left(\sum_{j=1}^{N} \beta_j
\bfy_j \right) \\  
&= g \cdot (f \bfw). \end{align*} 
Thus we have
shown that $f (g \cdot \bfw) =  g \cdot (f \bfw) $, so   the $G$ action
on $\frakf_{n,r}(K)$ commutes with 
$f$  as asserted.

Consider an  eigenvector $\bfz$ for $f$
with eigenvalue  $\alpha.$  Because 
$\bfz$ is in  $\ker (f - \alpha \Id),$ and $\alpha \in K,$
the vector $\bfz$ is a $K$-linear
combination of elements of the rational basis.  
For $g$ in $G,$ 
\[  f(g \cdot \bfz) = g \cdot f(\bfz) =  g \cdot \alpha \bfz = (g
\cdot \alpha) (g \cdot \bfz).\] Hence an  automorphism $g$ in $G$
sends  the $\alpha$-eigenspace  to the $(g \cdot
\alpha)$-eigenspace.  Therefore, Part \eqref{action-eigenspace} 
of the proposition holds.
\end{proof}

Now we use the action  defined in Definition \ref{action} 
to describe certain important kinds of
subspaces and ideals of Anosov Lie algebras. 

\begin{definition}\label{Ez} 
Let $G$ be a finite group that acts on the free nilpotent
Lie algebra $\frakf_{n,r}(K)$ over field $K,$ and let
 $\bfz \in \frakf_{n,r}(K).$    Define the
subspace $E_G^K(\bfz)$ of $\frakf_{n,r}(K)$ to be the 
$K$-span of the $G$ orbit of $\bfz:$ 
\begin{equation}\label{eg-def}
  E_G^K(\bfz) =  
\myspan_{K} \{ g \cdot \bfz \, : \, g \in G \}.\end{equation}
\end{definition}
 
\begin{example}\label{type-cn}  
Let $\frakf_{n,2}(K) = V_1(K) \oplus V_2(K)$ 
be the free two-step Lie algebra on $n$ generators over field $K.$ 
Let $\calC_1 = \{\bfz_i\}_{i=1}^n$ be a set of $n$ generators for 
$\frakf_{n,2}(K).$  
Let $G$ be the cyclic group  of order $n$ that acts on 
$\frakf_{n,2}(K)$ through the natural action of $G \cong C_n$ on $\calC_1.$
 For each $j = 2, \ldots, \lfloor n/2 \rfloor + 1,$
define the ideal $\fraki_j < V_2(K)$ by 
\begin{align*}
 \fraki_j = E_G^K([\bfz_j, \bfz_1]) 
= \myspan_{K} \{ [\bfz_{s},\bfz_{t}] \, : \, s - t = j - 1 \mod n  \}.
\end{align*} 
For example, when $n =4,$
\begin{align*}
  \fraki_2 &= \myspan_{K} \{ [\bfz_2, \bfz_1], [\bfz_3, \bfz_2],
[\bfz_4, \bfz_3], [\bfz_1, \bfz_4] \},
 \quad \text{and}\\
  \fraki_3 &= \myspan_{K} \{ [\bfz_3, \bfz_1], [\bfz_4, \bfz_2] \}.
\end{align*}
For distinct $j_1$ and $j_2$ in $\{1, \ldots, \lfloor n/2 \rfloor \},$ 
the subspaces 
$\fraki_{j_1}$ and $\fraki_{j_2}$ are independent, hence
 $V_2(K) = \oplus_{j=1}^{ \lfloor n/2 \rfloor}\fraki_j.$
When $n = n_1$ is odd, there are $\frac{1}{2}(n-1)$ such subspaces, 
each of dimension
$n.$  When $n = n_1$ is even, the
subspaces $\fraki_j, j =1, \ldots, n/2 - 1$ are of dimension $n,$
and the subspace $\fraki_{n/2}$ is of dimension $n/2.$

For any proper subset $S$ of $\{1, \ldots, \lfloor n/2 \rfloor\},$ 
 there is an
ideal $\fraki(S)$ of $\frakf_{n,2}(K)$ 
defined by $\oplus_{j \in S} \fraki_j,$ and this ideal defines a 
two-step Lie algebra 
$\frakn_S = \frakf_{n,2}(K)/\fraki_S.$
\end{example}

We define  ideals of free nilpotent Lie algebras arising from
group actions, as the Lie algebra $\frakn_S$ in the previous example
 arises from the action of a cyclic group. 
\begin{definition}\label{ideal-i} 
 Let $G$ be a finite group that acts on the free $r$-step nilpotent Lie algebra
$\frakf_{n,r}(K) = \oplus_{i=1}^r V_i(K)$ over field $K,$ where the field $K$ 
is an extension of $\boldQ.$
Let $L$ be another extension of $\boldQ,$ typically $\boldR$ or $\boldC.$
Suppose that Hall bases $\calB \subset \frakf_{n,r}(K)$ 
and $\calB^\prime \subset \frakf_{n,r}(L)$  define
rational structures on  $\frakf_{n,r}(K)$ and  $\frakf_{n,r}(L)$
respectively, and let $E \to E^L$ be the correspondence of rational invariant 
subspaces defined in Definition \ref{correspond}
 resulting from an identification of $\calB^\prime$ and $\calB.$  

A rational ideal of  $\frakf_{n,r}(L)$ generated by sets of the form
$(E_G(\bfw))^L,$ where $E_G^K(\bfw)$ is 
as defined in Definition \ref{Ez}, for  $\bfw \in \frakf_{n,r}(K),$
 is called {\em an ideal of type $G$ defined over $L$}.
We use $\fraki(G,\bfw)$ denote the ideal of $\frakf_{n,r}(L)$ 
generated by the subspace 
$(E_G^K(\bfw))^L.$ 

A nilpotent Lie algebra of form   $\frakf_{n,r}(L)/\fraki,$ where 
$\fraki$ is of type $G$ will be called a {\em 
nilpotent Lie algebra of type $G.$}
\end{definition}

Now we describe  ideals of symmetric type for two- and three-step
free nilpotent Lie algebras.  
\begin{example}\label{type-sn}
Suppose that $\frakf_{n,r}(K)$ has generating set 
$\calC_1 = \{\bfz_j\}_{j=1}^n$ and that the action of $G \cong S_n$
 on $\frakf_{n,r}(K)$ is defined by permuting elements
$\bfz_1, \ldots, \bfz_n$ of the generating set $\calC_1.$  
Let $\calC = \cup_{i=1}^r \calC_i$ be the Hall basis determined by 
$\calC_1.$  

For any $\bfw = [\bfz_i,\bfz_j]$ in $\calC_2,$
the subspace $E_G^K(\bfw)$ is all of $V_2(K):$
\[ E_G^K(\bfw) = \myspan_{K} \{ [\bfz_j,\bfz_i]\}_{1 \le i < j \le n} = V_2(K).\]
Hence, an ideal $\fraki < \frakf_{n,r}(K)$
 of type $S_n$ that intersects $V_2(K)$ nontrivially must contain
all of $V_2(K).$

When $r \ge 3,$ there are two sets of form  $E_G^K(\bfw)$
with $\bfw \in \calC_3,$ 
the subspaces $F_1$ and $F_2$ defined
in Equation \eqref{F defs}:
\[ F_1(K) = E_G^K([[\bfz_2,\bfz_1], \bfz_1]) \quad
\text{and} \quad F_2(K) = E_G^K([[\bfz_2,\bfz_1], \bfz_3]). \]
Therefore, for any ideal $\fraki$ of type $S_n,$ 
 the subspace $\fraki \cap V_3(K)$ is one of the following: 
$\{0\},$ $F_1(K),$ $F_2(K),$ or  $V_3(K).$
\end{example}
Next is an example showing nilpotent Lie algebras 
arising from  dihedral groups.  
\begin{example}\label{D2n} Let $A$ be a semisimple 
matrix in $GL_n(\boldZ)$ whose
characteristic polynomial has   splitting field $K$
and Galois group  $G$ isomorphic to the
dihedral group  $D_{2n}$ of order $2n.$ Let
$f$ be the automorphism of  $\frakf_{n,2}(K)$ induced by $A$ and a Hall
basis $\calB.$ Let  $\bfz_1,  \ldots, \bfz_n$ denote a set of
 eigenvectors of $f|_{V_1(K)}$ spanning 
$V_1(K)$ compatible with the rational structure, 
and let  $\alpha_1,  \ldots, \alpha_n$ denote
corresponding eigenvalues.

The group  $D_{2n}$ is isomorphic to  the group of symmetries  of a
regular $n$-gon.   Enumerate the vertices of such an $n$-gon in
counterclockwise order so that $D_{2n} \cong \la r, s\ra,$ where $r$ is
counterclockwise rotation by $2\pi/n$ and $s$ is reflection 
through a line through center of the $n$-gon and the first vertex.   
Let  $X_n$ be  the
complete graph on $n$ vertices obtained by adding edges connecting all
distinct vertex pairs of the $n$-gon.   Identify the roots of
$p_1$ with the $n$ vertices and   the roots of $p_2$ with the
$(\begin{smallmatrix} n \\ 2 \end{smallmatrix})$  edges in such a way
that the eigenvalue  $\alpha_i \alpha_j$ corresponds the to edge
connecting vertices corresponding to eigenvalues $\alpha_i$ and
$\alpha_j.$ 
 The $G$ action on $\frakf_{n,2}(K)$ can then be visualized through the
$D_{2n}$ action on the graph $X_n.$  For example, the $G$ orbit of the
eigenvector  $[\bfz_2,\bfz_1]$ is
\[ G \cdot [\bfz_2,\bfz_1] =  \{ [\bfz_2,\bfz_1], [\bfz_3,\bfz_2],
\ldots, [\bfz_1,\bfz_n] \},\] corresponding to the $n$ ``external''
edges of the graph $X_2.$ The subspace $E_G^K([\bfz_2,\bfz_1])$ defined
by Definition \ref{Ez} is the $K$-span of this set, an
$n$-dimensional subspace of
$V_2(K).$ Other orbits depend on the value of $n:$ if $n=3$ there are
no other orbits, and if $n=4$ or $5,$ there is one more orbit  coming
from ``interior'' edges on the graph, 
 yielding a subspace  $E_G^K([\bfz_3,\bfz_1])$
 of $V_2(K)$ that is complementary to 
$E_G^K([\bfz_2,\bfz_1]).$  When $n \ge 6,$ there are two more  orbits
coming from interior edges.
\end{example}

\section{Rational invariant subspaces}\label{Rational invariant subspaces}

Given an automorphism of a free nilpotent Lie algebra, the 
next  theorem  describes how orbits of  the $G$ action  on
$\frakf_{n,r}(\boldR)$ relate to the factorization of the polynomials  $p_1, \ldots,
p_r.$ These restrictions on factorizations yield restrictions
on the existence of  Anosov quotients.  Roughly speaking, 
when the Galois group of $p_1$ is highly transitive,  the 
rational invariant subspaces for associated Anosov automorphisms
tend to be big also, and when the group is small, the rational
invariant subspaces are small.
 The larger rational invariant subspaces are, the
fewer Anosov quotients there may be.  The field $L$ in the theorem
is typically $\boldR$ or $\boldC.$

\begin{thm}\label{actions} Let $A$ be a semisimple matrix in  $GL_n(\boldZ).$
   Let $(p_1, \ldots, p_r)$ be the
$r$-tuple of polynomials associated to $f,$  let 
$K$ be the splitting field of $p_1$ and let $G$ be the Galois group for $K.$
Let $f$ be the 
 semisimple automorphism of  $\frakf_{n,r}(K)$
defined by Hall basis $\calB = \cup_{i=1}^r \calB_i$ of   $\frakf_{n,r}(K)$
and the matrix $A.$
  Let $\calC = \cup_{i=1}^r
\calC_i$ be the Hall basis for $\frakf_{n,r}(K)$ determined by a set of
$\calC_1 =  \{\bfz_j\}_{j=1}^{n}$ of eigenvectors for $f|_{V_1(K)}$
that is compatible with the rational structure defined by $\calB.$

For all $i = 1, \ldots, r,$  the vector subspace $V_i(K)$ of  $\frakf_{n,r}(K)$ 
 decomposes as
the direct sum of rational invariant subspaces of the form $E_G^K(\bfz),$
where $\bfz \in \calC_i,$  and $E_G^K(\bfz)$ is as defined in
Definition \ref{Ez}.  The characteristic polynomial $p_E$ for
 the restriction of $f$ to $E_G^K(\bfz)$ is of form $p_E = r^s,$
where $r$   is  a polynomial that is
 irreducible over $\boldZ.$  

 Suppose that the field $L$ is an extension of $\boldQ,$
and that $f$ is a
 semisimple automorphism of  $\frakf_{n,r}(L)$
defined by Hall basis $\calB^\prime = \cup_{i=1}^r \calB_i^\prime$
and the matrix $A.$   
Since the subspaces  $E_G^K(\bfz)$ of  $\frakf_{n,r}(K)$ are rational, 
for  all $i = 1, \ldots, r,$  there is a 
decomposition $V_i(L) = \oplus (E_G^K(\bfz))^L$
of $V_i(L) < \frakf_{n,r}(L)$ into rational $f$-invariant
subspaces, through the correspondence defined in Definition 
\ref{correspond} induced by the identification of the 
rational Hall bases $\calB$
and $\calB^\prime$ of $\frakf_{n,r}(K)$ and $\frakf_{n,r}(L)$.
\end{thm} 
 
We illustrate the theorem by considering a special case
of Example \ref{D2n}.
\begin{example}\label{D10} Let $A$ be a semisimple hyperbolic matrix in
$GL_5(\boldZ)$ such that the splitting field $K$ for the
characteristic polynomial $p_1$ for $A$ has Galois group  $G$ isomorphic to the
dihedral group  $D_{10}$ of order $10.$  Let $f_K$ be the automorphism
of  $\frakf_{5,2}(K) = V_1(K) \oplus V_2(K)$ 
induced by $A$ and a Hall basis $\calB.$ Let
$\calC_1 = \{\bfz_1, \ldots, \bfz_5\}$
  be the set of eigenvectors of
$f_K|_{V_1(K)}$ and let $\alpha_1, \ldots, \alpha_5$ denote the corresponding
eigenvalues. Let $\calC = \calC_1 \cup \calC_2$ be the Hall basis
of $\frakf_{5,2}(K)$ determined by 
$\calC_1.$ 

We saw in Example \ref{D2n} that the $D_{10}$ action on $\frakf_{5,2}(K)$ 
 has two
orbits each of $K$-dimension five.  By Theorem \ref{actions}, these
two orbits yield rational $f_K$-invariant subspaces 
\[ \fraki_1(K) =  E_{G}([\bfz_2, \bfz_1]), 
\quad \text{and}
\quad   \fraki_2(K) =  
E_{G}([\bfz_3, \bfz_1])\]
and a decomposition  $V_2(K) = \fraki_1(K) \oplus \fraki_2(K) $
of $V_2(K) < \frakf_{5,2}(K)$ into rational invariant subspaces.   

Let $f_{\boldR}$ be the automorphism of $\frakf_{5,2}(\boldR)$ induced by 
$A$ and Hall basis $\calB^\prime = \cup_{i=1}^r \calB_i^\prime$
of  $\frakf_{5,2}(\boldR).$
Letting $L = \boldR$ in the Theorem \ref{actions}, 
and using the correspondence between  $\frakf_{5,2}(K)$
and  $\frakf_{5,2}(\boldR)$ as  in Definition \ref{correspond},
we get a rational $f_R$-invariant decomposition of 
 $V_2(R) < \frakf_{5,2}(\boldR)$  into rational $f_\boldR$-invariant subspaces
$ \fraki_1(\boldR) = (\fraki_1(K))^\boldR  $
and $\fraki_2(\boldR) = (\fraki_2(K))^\boldR .$

The
polynomial $p_2(x)$ factors as $p_2 = r_1 r_2,$
where the two quintic factors $r_1$ and $r_2$
 are characteristic polynomials of the restriction of $f$ to the 
rational invariant ideals 
$\fraki_1(K)$ and $\fraki_2(K)$
 respectively.    Neither 
$r_1$ nor $r_2$ has
roots of modulus one by Remark \ref{roots of modulus one}.
By Lemma \ref{how it can factor}, Part \ref{factor-degree},
both $r_1$ and $r_2$ are irreducible. 

Thus, the only
 ideals $\fraki$ of $\frakf_{5,2}(\boldR)$
satisfying the Auslander-Scheuneman conditions 
for $f: \frakf_{n,r}(\boldR) \to \frakf_{n,r}(\boldR)$ and 
$\calB$
and defining a two-step Anosov quotient $\frakf_{5,2}/\fraki$
of type $(5,n_2)$
are the trivial ideal, the ideal $\fraki_1(\boldR)$ 
and the ideal $\fraki_2(\boldR).$

Now we'd like to write out the ideals  $\fraki_1^\boldR$ 
and  $\fraki_2^\boldR$ in terms of generators and relations.  
We need to consider two cases.  In the first case, all of the eigenvalues
of $f_\boldR$ are real, so that the 
quotient algebras $\frakn_1 = \frakf_{5,2}/\fraki_1(\boldR)$
and $\frakn_1^\prime= \frakf_{5,2}/\fraki_2(\boldR)$
may be written with generators $\bfz_1,\bfz_2,\bfz_3,\bfz_4,\bfz_5$
in $\frakf_{5,2}(\boldR)$
and relations 
\[ [\bfz_1,\bfz_2]= [\bfz_2,\bfz_3]= [\bfz_3,\bfz_4]= 
[\bfz_4,\bfz_5]= [\bfz_5,\bfz_1]=0  \]
for the first Lie algebra $\frakn_1,$
and  relations 
\[  [\bfz_1,\bfz_3]= [\bfz_2,\bfz_4]= [\bfz_3,\bfz_5]= [\bfz_4,\bfz_1]= [\bfz_5,\bfz_2]=0 \]
for the second Lie algebra.  These Lie algebras are clearly isomorphic. 

In the second case, there is an eigenvector
$\bfz_1$ with a real eigenvalue, and there 
are two complex eigenvalue pairs yielding eigenvector pairs
$\bfz_2, \bfz_3 = \bfx_2 \pm i \bfy_2$ and 
 $\bfz_4, \bfz_5= \bfx_3 \pm i \bfy_3,$ where 
the vectors $\bfx_i, \bfy_i < \frakf_{5,2}(\boldR)$ ($i=1,2$).  
(If there were
only one complex eigenvalue pair, the Galois group would be $S_5.$)
The reader may check that the ideal  $\fraki_1^\boldR$ is generated by
the elements 
\[ [\bfz_1, \bfx_2], [\bfx_2, \bfx_3] - [\bfy_2,\bfy_3], 
[\bfx_2,\bfy_3] + [\bfy_2,\bfx_3],  [\bfx_3, \bfy_3], [\bfz_1, \bfy_2] \]
of $V_2(\boldR) < \frakf_{5,2}(\boldR).$
The ideal  $\fraki_2^\boldR$ yields another ideal isomorphic to
 the first.  These ideals yield isomorphic Lie algebras 
$\frakn_2$ and $\frakn_2^\prime.$

It can be shown that $\frakn_1$ and $\frakn_2$ are not isomorphic, 
because the the  $D_{10}$ symmetry of $\frakf_{5,2}(K)/ \fraki_1$
is preserved when moving to $\frakn_1,$ 
 but it is lost when moving to 
$\frakn_2.$  (In particular, the 
$\bfz_1$ coset in $\frakf_{5,2}(\boldR)/\fraki_2^\boldR$
is the unique element having a three-dimensional
centralizer.)

In summary,  
in addition to  $\frakf_{5,2}(\boldR),$
 there are exactly two isomorphic two-step
nilpotent quotients, both of type $(5,5)$ to which
$f$ descends as an Anosov map, and there are exactly 
two nonisomorphic Anosov Lie algebras of type $(5,5)$ with
Anosov automorphisms yielding Galois group $D_{10}.$   
\end{example}

\begin{proof}[Proof of Theorem \ref{actions}.]
 Fix an element $\bfw$ in the 
basis $ \calC_i$ for $V_i(K) < \frakf_{n,r}(K).$  
By Proposition \ref{p-properties}, 
it is an eigenvector for $f;$ let $\alpha$ 
denote its eigenvalue.  When represented with respect
to the rational basis $\calB_i$ for $V_i(K),$ 
the vector $\bfw$  has coordinates in
$K^{n_i},$ where $n_i = \dim V_i.$ Let $E_G^K(\bfw)$ be the subspace of $V_i(K)$
generated by $\bfw$ and $G$ as in Definition \ref{Ez}.

First we show that $E_G^K(\bfw)$ is invariant under $f.$ By Proposition 
\ref{action-properties}, part \eqref{action-eigenspace}, for all
$g$ in $G,$ the vector $g \cdot \bfw$ is an eigenvector with
eigenvalue $g \cdot \alpha.$  An element
$\bfu$ of $V_i(K)$ is in $E_G^K(\bfw)$ if and only if it is of the form
$\bfu = \sum_{g \in G} c_g \, (g \cdot \bfw),$ where $c_g \in K$
for $g \in G.$ Then 
  \[ f(\bfu) = \sum_{g \in G} c_g \, f(g \cdot \bfw) =  \sum_{g \in G}
c_g \, (g \cdot \alpha) \,  g \cdot \bfw  ,\] so $f(\bfu)$ is also in
$E_G^K(\bfw).$ 

The vector space  $E_G^K(\bfw)$ is spanned by vectors with 
coordinates in $K,$ so there exists
a polynomial function $\phi : V_i(\boldR) \cong \boldR^{n_i} \to \boldR$
with coefficients in $K$  such that $E_G^K(\bfw)$ is the zero set of
$\phi.$  Because  $E_G^K(\bfw)$ is invariant under the $G$ action,
for all $g$ in $G,$
$E_G^K(\bfw)$ is also the  zero set of the function  $\phi_g(x) = \phi (g
\cdot x).$  Let $\bar \phi = \prod_{g \in G}
\phi_g.$ The function $\bar \phi$  has
rational coefficients, so $E_G^K(\bfw)$ is a rational $G$-invariant subspace.
 
By Maschke's Theorem, the subspace
$V_i(K)$ may be written as the direct sum of subspaces of form 
$E_G^K(\bfw).$

Now we show that the characteristic polynomial for the
restriction of $f$ to $E_G^K(\bfw)$  is a power of an irreducible.   Let
$q$ denote the  minimal polynomial for $\alpha.$  By Proposition
\ref{p-properties}, the splitting field  $\boldQ(q)$
is intermediate to  $\boldQ$ and the splitting field $\boldQ(p_1)$ of $p_1.$
 The $G$-orbit of $\alpha$ is precisely the set of roots of $q.$  
 Since $g \cdot \alpha$ is a root of $q$ for all $g \in G,$
 the space $E_G^K(\bfw)$ is contained in the direct sum of the
eigenspaces for eigenvalues $g \cdot \alpha, g \in G.$ Therefore, the
characteristic polynomial $p_E$ for the restriction of $f$ to $E_G^K(\bfw)$ is a
power of the irreducible polynomial $q.$ 
\end{proof}

We will describe a some rational invariant subspaces
that exist for any automorphism of a free nilpotent Lie algebra
preserving a rational structure.
First we need to make a definition.
\begin{definition}\label{j-def} Let $K$ be a field. 
Let $\calC_1 = \{\bfz_j\}_{j=1}^n$ be  a 
generating set for the free $r$-step nilpotent 
Lie algebra $\frakf_{n,r}(K),$ and let $\calC$ be the associated Hall
basis.  
Define the ideal $\frakj_{n,r}$ of $\frakf_{n,r}(K)$ to be the ideal
generated by all elements $\bfw$ of the Hall basis $\calC$
 having the property that there is a single number $k$ such that 
for all $j = 1, \ldots, n,$ the letter $\bfz_j$
 occurs exactly $k$  times in the Hall word $\bfw.$  
\end{definition} 

 For example, when $n = 3,$ the ideal $\frakj_{3,2} < \frakf_{3,2}(K)$
is $\{0\},$ and
the ideal $\frakj_{3,3} < \frakf_{3,3}$ is
given by 
\[ \frakj_{3,3} = \myspan_K \{ [[\bfz_2,\bfz_1],\bfz_3],
 [[\bfz_3,\bfz_1],\bfz_2]\} =   F_2(K) < V_3(K),\]
where $F_2(K)$ is as defined in Equation \eqref{F defs}, and
the ideal  $\frakj_{4,3} < \frakf_{4,3}(K)$ is
given by 
\[ \frakj_{4,3} = \frakj_{3,3} \oplus [\frakj_{3,3},\frakf_{4,3}],\]
where we map  $\frakj_{3,3}$ into $\frakf_{4,3}$ in the natural way.  

\begin{remark}\label{j in i}
Since the product of the roots of an Anosov polynomial is 
always $\pm 1,$  any ideal $\fraki < \frakf_{n,r}$ 
satisfying the Auslander-Scheuneman conditions for some $f$
must contain the ideal $\frakj_{n,r}$ (defined relative to 
an eigenvector basis). 
\end{remark}

\begin{proposition}\label{Fdecomp} 
Let $A$ be a semisimple 
matrix in  $GL_n(\boldZ)$ whose characteristic polynomial
has splitting field $K.$
Let $\frakf_{n,r}(K) = \oplus_{i=1}^r V_i(K)$ be the free $r$-step nilpotent
Lie algebra on $n \ge 3$ generators over $K$, 
endowed with the rational structure
defined by a Hall basis $\calB.$  Let $f$ be the semisimple automorphism
of $\frakf_{n,r}(K)$ defined by the matrix $A$  and the Hall
basis $\calB.$
Let $\calC$ be the Hall basis of  $\frakf_{n,r}(K)$ determined by 
a set of eigenvectors 
$\calC_1$ for $f|_{V_1(K)}$ that is compatible with the rational structure.
 
  The ideal $\frakj_{n,r}$ defined in Definition \ref{j-def}
is a rational invariant subspace; and  
  when $r \ge 3,$ the subspace $V_3(K)$ is the
direct sum $F_1(K) \oplus F_{2}(K)$ of rational invariant subspaces 
where $F_1(K)$ and $F_2(K)$ are as in Equation \eqref{F defs},
while $F_2(K)$ decomposes as the direct sum 
 $F_2(K) = F_{2a}(K) \oplus F_{2b}(K)$ where 
$F_{2a}(K)$ and $F_{2b}(K)$ are rational invariant subspaces each of dimension
$(\begin{smallmatrix} n \\ 3 \end{smallmatrix})$. 
The characteristic polynomial for the restrictions of $f$ to
$F_1(K)$ is $q_1$ and the characteristic polynomials for the restriction 
of $f$ to $F_{2a}(K)$ and $F_{2b}(K)$ are both $q_2.$ 
\end{proposition} 

\begin{proof}  Let $G$ denote the Galois group of 
the polynomial $p_1$ associated to $f.$  Suppose that $\bfw$ is
a $k$-fold bracket of elements in $\calC_1 = \{\bfz_j\}_{j=1}^{n_1}.$

Recall from 
Example \ref{2,3-Hall words} 
that $\calC_3$ is the union of the set $\calC_3^\prime$
of standard Hall monomial of the first
type and the set $\calC_3^{\prime \prime}$ of
  standard Hall monomials of the second type.  Let $g \in S_n.$
It is easy to see that if
 $\bfw \in \calC_3^\prime$   then  $g \cdot \bfw \in \calC_3^\prime$
or  $-(g \cdot \bfw) \in \calC_3^\prime,$
and if  $\bfw \in \calC_3^{\prime \prime}$
   then  $g \cdot \bfw \in \calC_3^{\prime \prime},$ 
$-(g \cdot \bfw) \in \calC_3^{\prime \prime},$ or
  $g \cdot \bfw$ is a linear  combination of elements of 
$\calC_3^{\prime \prime}$ through the Jacobi Identity. 
 The action of the  group $G$ on
$\frakf_{n,r}(K)$ therefore preserves the subspaces $F_1$ and $F_2,$ so
that if $\bfw$ in $\calC_2$ is in $F_1(K)$ or $F_2(K),$ then $E_G^K(\bfw) <
F_1$ or $E_G^K(\bfw) < F_2(K)$ respectively.  The space $F_1(K)$
is the sum of the rational invariant spaces $E_G^K(\bfw)$ as $\bfw$ varies
over elements of $\calC_3^\prime$, so is rational and
invariant.  By the same reasoning $F_2(K)$ is 
rational and $f$-invariant also. 
  
The characteristic polynomial for the restriction of  $f$ to
$F_2(K)$ is $q_2^2,$ where $q_2$ is as defined in Equation \eqref{q1,q2}.
The pair of elements of form  $[[\bfz_i,\bfz_j],\bfz_k]$ and
$[[\bfz_k,\bfz_j],\bfz_i],$
where $1 \le j < i < k \le n,$ in  $ \calC_3$ have the same eigenvalue,
$\alpha_i \alpha_j \alpha_k,$ where $\alpha_i, \alpha_j, \alpha_k$ are
the eigenvalues of $\bfz_i, \bfz_j,$ and $\bfz_k$ respectively.   
These yield one basis vector for $F_{2a}$ and one basis vector for 
$F_{2b}.$
 Each factor $q_2$ in the characteristic polynomial 
$q_2^2$ for $F_2(K)$ yields one  rational invariant subspace of
$F_2(K)$ of dimension 
$\deg q_2 = (\begin{smallmatrix} n \\ 3 \end{smallmatrix}),$
each spanned by elements of $\calC_3^{\prime \prime}.$
Call these subspaces $F_{2a}(K)$ and $F_{2b}(K),$ so  
  $F_{2}(K) = F_{2a}(K) \oplus F_{2b}(K).$

Similarly, the set of $k$-fold brackets $\bfw$ of $\calC$ having the
property that each element $\bfz_i$ occurs the same number of times
in $\bfw$ is clearly invariant under the action of $G$, so the
ideal that it generates, $\frakj_{n,r},$
is $G$-invariant, hence rational.  
\end{proof}

The following elementary proposition yields restrictions on 
 possible dimensions of rational invariant
subspaces for semisimple automorphisms of  nilpotent Lie algebras. 

\begin{proposition}  Let $\frakf_{n,r}(\boldR)$ be a 
 free nilpotent Lie algebra, and
let $f : \frakf_{n,r}(\boldR) \to \frakf_{n,r}(\boldR)$
 be a semisimple automorphism 
 defined by a matrix $A$ in $GL_n(\boldZ)$ and a Hall basis 
 $\calB$ of $\frakf_{n,r}(\boldR).$  
Suppose that the characteristic polynomial $p_1$ of $A$ 
is irreducible  with  Galois group $G.$ 
Let $m$ be the dimension of a minimal nontrivial rational invariant subspace
 $E < \frakf_{n,r}(\boldR)$ for $f.$ 
Then $G$ has a  normal subgroup $N$  such that $G/N$ acts
faithfully and  transitively on a set of $m$ elements. 
\end{proposition}

Note that the subspace $E$ is one-dimensional if and only if $N = G.$
\begin{proof} Suppose that $E$ is a minimal nontrivial invariant
subspace  of dimension $m.$ 
The characteristic polynomial $p_E$ for the restriction of 
$f$ to $E$ is irreducible. Since
the roots of $p_E$ are contained in the splitting field $\boldQ(p_1),$
the Galois group for  $p_E$ 
 is the quotient  $G/N$ of $G$ by the normal
subgroup of elements of $G$ fixing $\boldQ(p_E).$
 The group $G/N$ acts faithfully and 
transitively on the $m$ roots of  $p_E$  since it the  Galois group of
an irreducible polynomial. 
\end{proof}

The previous proposition can be used to find all possible dimensions
of minimal nontrivial rational invariant subspaces for any
Anosov automorphism whose associated  Galois group is some 
fixed group $G.$  One simply needs to find all numbers $m$ such that
there exists a normal subgroup 
$N$ of $G$ and there exists a faithful transitive
action of $H = G/N$ on a set of $m$ elements.  
Every faithful transitive action of a group $H$ on a set $X$ is conjugate
to the action of $H$ on the cosets $X^\prime = \{hK\}_{h \in H}$ of a subgroup
$K$ of $H$ such that $K$ contains no nontrivial  normal subgroups of
$H.$ To find all faithful transitive actions of a 
group $H = G/N,$  one must 
list all subgroups of $H$ and eliminate any that contain
nontrivial normal subgroups.  The cardinalities  $|H| / |K|$ of the set
$\{hK\}_{h \in H}$ are admissible  values for the cardinality $m$ of a
set $X$ on which $H$ acts faithfully and transitively.  In our situation,
where $G$ is the Galois group of $p_1,$ 
the number $m$ could be the degree of a polynomial 
having Galois group $G/N,$  and  $m$ could be the
 dimension of a rational invariant subspace of the
corresponding automorphism of $\frakf_{n,r}(\boldR).$

In Table  \ref{degrees} we analyze the possible dimensions of
rational invariant subspaces of Anosov Lie algebras $\frakn$ for which
$p_1$ is irreducible and degree 3 or 4.  
The first two columns in the table  give all possible Galois
groups $G$ for irreducible polynomials degrees three and four,
grouped by degree.   
 The fourth and fifth columns list the isomorphism class of 
proper normal subgroups $N$ of each group $G,$ and the quotients
$G/N.$  (Since $\frakn$ is Anosov, there are no one-dimensional rational 
invariant subspaces, and we omit the case $N = G$.) 
The quotient groups $G/N$ are potential 
 Galois groups for
characteristic polynomials of rational invariant subspaces
of an Anosov Lie algebra with polynomial $p_1$ having Galois group $G.$
 The last column gives the cardinalities of sets
 on which each $G/N$ can act faithfully and transitively, found by 
the procedure described above.  
 The numbers in the last column are listed in the same order
as the subgroups in the second column, with the numbers  for different
subgroups separated by semicolons.  
 In the third column
we show when the roots of a polynomial with given Galois group
must have full rank by  Proposition \ref{full rank}. 
When the set of roots has full rank, 
some possibilities for the normal subgroup $N$ may be prohibited
by Lemma \ref{how it can factor}:  these subgroups and the corresponding
dimensions are indicated in the table with asterisks.

\renewcommand{\arraystretch}{2}
\begin{table}
\begin{tabular}{|c|c||c|l|l|l|} \hline $\deg p_1$ & $G$ & full rank? 
&$N \ne G$ &
$G/N$ & dimension $m$ of  $E$ \\ 
\hline \hline
 $3$ & $S_3$ & yes & $C_3^\ast, \{1\}$ & $C_2^\ast, S_3$ & $2^\ast;3,6$ \\ 
\hline 
$3$ & $C_3$ & yes & $\{1\}$ & $C_3$ & $3$ \\ 
\hline \hline 
$4$ & $S_4$ & yes & $A_4^\ast, V_4^\ast, \{1\}$ & $C_2^\ast,S_3^\ast,S_4$ & 
       $2^\ast; 3^\ast, 6^\ast; 4, 6, 8, 12, 24$ \\ 
\hline $4$ & $A_4$ & yes & $V_4^\ast, \{1\}$ &
$C_3^\ast,A_4$  & $3^\ast; 4, 12 $ \\ 
\hline $4$ & $D_8$ & no & $C_4, C_2, \{1\}$ & $C_2, V_4, D_8$ & $2; 4;  4, 8$ \\ 
\hline $4$ & $C_4$ & no & $C_2, \{1\}$ & $C_2, C_4$ & $2;4$ \\ 
\hline $4$ & $V_4$ & no & $C_2, \{1\}$ & $C_2, V_4$ & $2;4$ \\
\hline
\end{tabular} \bigskip
\caption{\label{degrees} Possible dimensions $m > 1$ for
 rational invariant subspaces $E$ for an
 Anosov automorphism  of an $r$-step nilpotent
Lie algebra $\frakn = \frakf_{n,r}(\boldR)/\fraki$
 when $n = 3$ or $4$ and $p_1$ is irreducible.  An
asterisk indicates that the marked values of $N, G/N$ and  
$\dim E$ cannot occur by Lemma \ref{how it can factor}, Part \ref{N-not-trans}.   }
\end{table}
From the table we obtain the following corollary to Theorem 
\ref{actions}.

\begin{corollary}\label{dimensions}
Let $f$ be a semisimple automorphism of $\frakf_{n,r}(\boldR)$ induced by 
a hyperbolic matrix in $GL_n(\boldZ)$ and a Hall basis $\calB.$
Let $(p_1, \ldots, p_r)$ be the $r$-tuple of polynomials associated to 
$f.$  If $p_1$ is irreducible, then the dimension 
of any minimal nontrivial invariant subspace of $\frakf_{n,r}(\boldR)$ is $3$ or $6$ if
$n=3$ and is one of $2, 4,  6, 8, 12, 24$ if $n =4.$
\end{corollary}

\section{Automorphisms with cyclic and symmetric  Galois
 groups}\label{symmetric-cyclic}
 In this section we
 use Theorem
\ref{actions}  to analyze the structure of Anosov Lie algebras whose
associated polynomial $p_1$ has either a small Galois group, such as a
cyclic group,  or a large Galois group, such as a symmetric group.
The following theorem describes Anosov automorphisms associated to
 Galois groups whose actions on the roots of $p_1$ is highly transitive.

\begin{thm}\label{Sn} Suppose that $\frakn$ is a real $r$-step Anosov Lie 
algebra admitting an Anosov automorphism
defined by  a semisimple matrix $A$ in $GL_n(\boldZ),$ a Hall basis
$\calB,$ and an ideal $\fraki < \frakf_{n,r}(\boldR)$ 
satisfying Auslander-Scheuneman conditions.
 Suppose that the polynomial $p_1$ associated to $f$ is irreducible with 
Galois group $G.$ 
  Let $(p_1, \ldots, p_r)$ be the $r$-tuple of
polynomials associated to $p_1.$
\begin{enumerate}
\item{If the action of $G$ on the roots of $p_1$ is two-transitive, then  
 the polynomial $p_2$ is irreducible and Anosov,  and 
if $r=2,$ then $\frakn$ is isomorphic to the free nilpotent
algebra $\frakf_{n,2}(\boldR).$}\label{2,2}
\item{If the action of $G$ on the roots of $p_1$ is three-transitive 
and $r=3,$ then $\frakn$ is isomorphic to 
$\frakf_{n,3}(\boldR)/\fraki,$ where $\fraki$ is trivial or 
a sum of $F_1(\boldR),$  $F_{2a}(\boldR),$  and $F_{2b}(\boldR),$ 
where $F_1(\boldR)$
is as defined in Equation \eqref{F defs}, 
and $F_{2a}(\boldR)$ and  $F_{2b}(\boldR)$ are as
in Proposition \ref{Fdecomp}.  
If $n=3,$ then $\fraki$ contains $F_2(\boldR).$
}
\item{A Lie algebra 
 $\frakf_{n,r}(\boldR)/\fraki$  of type  $S_n$  is Anosov 
 so long as the  ideal  $\fraki$  contains the 
ideal $\frakj_{n,r}$ as defined in Definition \ref{j-def} }\label{sn is this}
\end{enumerate}
\end{thm}

Note that when $G$ is $S_n,$ the Anosov Lie algebra need
not be of type $S_n$ as in Definition \ref{ideal-i}: the Lie algebra 
$\frakf_{n,3}/F_{2a}(\boldR)$ is not of type $S_n$ although it admits
an Anosov automorphism with symmetric Galois group. 
\begin{proof}
 Let $\alpha_1, \ldots, \alpha_{n}$ denote the roots of
$p_1,$ and let $\calC_1 = \{\bfz_1, \ldots, \bfz_n\}$ be a set of 
corresponding eigenvectors
of $f|_{V_1(K)}$ that is compatible with the rational structure
defined by $\calB.$  
Let $\calC = \cup_{i=1}^r \calC_{i}$ be the Hall basis
of $\frakf_{n,r}(K)$ determined by  $\calC_1.$

Suppose that the action of 
the Galois group $G$ on $\{\alpha_1, \ldots, \alpha_n\}$ 
is two-transitive. If the action of $G$ on the roots of
$p_1$ is two-transitive, then the set of roots of $p_1$
has full rank by Lemma \ref{how it can factor}.
   Since
the action of the group $G$ sends an eigenvector $\bfz_i$
to a multiple of another eigenvector $\bfz_j,$  the $G$ action
sends an element of 
 $\calC_2 = \{ [\bfz_k, \bfz_j]\}_{1 \le j < k \le n}$  to a scalar multiple of
an element of $\calC_2.$
Because the  action is doubly transitive,
 for any $\bfw$ in $ \calC_2,$ the rational invariant 
subspace $E_G^K(\bfw)$ is all of $V_2(\boldR).$  By Theorem
\ref{actions}, the characteristic polynomial $p_2$ 
 for the restriction of $f$  to $V_2(\boldR)$ is  a
 power of an irreducible polynomial.   But actually, $p_2$ itself
is irreducible by  Part \eqref{fr-p2-irr} of 
Lemma \ref{how it can factor}.
Thus, the only proper rational invariant subspace of $V_2(\boldR)$ is 
trivial, and when $r=2,$ the only two-step Anosov quotient of
  $\frakf_{n,2}(\boldR)$ is itself.
This proves the first part of the theorem.  

Now suppose $r=3$ and that the action of 
the Galois group $G$ on $\{\alpha_1, \ldots, \alpha_n\}$ 
is three-transitive. Then $G$ is two-transitive, and by the
 argument above, 
the only proper rational invariant subspace of $V_2(\boldR)$ is $\{0\}.$
Recall from Proposition \ref{Fdecomp} that $V_3(\boldR)$ is the direct sum 
$V_3(\boldR) = F_1 \oplus F_{2a} \oplus F_{2b}$ of the rational invariant
  subspaces
$F_1, F_{2a},$ and $F_{2b}$ where $F_2 = F_{2a} \oplus F_{2b},$ and 
the  characteristic polynomial for the restrictions of $f$ to
$F_1$ is $q_1$ and the characteristic polynomials for the restriction 
of $f$ to $F_{2a}$ and $F_{2b}$ are both $q_2.$ 
  By the three-fold transitivity of $G$, the subspaces 
$F_1$ and $F_{2}$ are all single $G$-orbits; hence
$q_1$ and $q_2^2$ are powers of  irreducibles.  But by 
Lemma \ref{how it can factor}, when $n\ge 3,$
 $q_1$ is irreducible, so $F_1$ has 
no nontrivial rational invariant subspaces, and when $n>3,$ the polynomial
$q_2$ is irreducible; hence the subspaces $F_{2a}$ and $F_{2b}$
are minimal nontrivial invariant subspaces.    
If $n=3,$ then $\alpha_1 \alpha_2 \alpha_3 =\pm 1, $ so  $q_2(x) =
x \pm 1,$ and $\fraki$ must contain $F_2.$ 
 Therefore, in order for an ideal $\fraki$ of $\frakf_{n,3}(\boldR)$
 to satisfy the Auslander-Scheuneman conditions relative
to $f$, it is necessary for  $\fraki$ to be a sum  of $\{0\}, F_1, F_{2a}$
 $F_{2b}.$ Thus, the
second part of the theorem holds.  

Now we consider the case that $G$ is symmetric. 
Assume that $\fraki_0$ is an ideal of
 $\frakf_{n,r}(\boldR)$  of type  $S_n$ relative to some Hall
basis $\calD,$ and $\frakf_{n,r}(\boldR)/\fraki_0$ is a Lie
algebra of type $S_n.$  Assume also that 
$\fraki_0$ contains $\frakj_{n,r}.$  We need to show that $\fraki_0$ is 
satisfies the Auslander-Scheuneman conditions relative to some automorphism
$f$ of  $\frakf_{n,r}(\boldR).$ 

 Let $A$ be the companion matrix to an
Anosov polynomial with Galois group $S_n,$ such as $p_1(x) = x^n - x - 1$
as in Proposition \ref{existence}.
 Together, $A$ and a set of generators $\calB_1$ determine an
automorphism $f$ of $\frakf_{n,r}(\boldR)$ that is rational relative to the 
Hall basis $\calB$ determined by $\calB_1.$  Let $\calC_1$ denote
the set of eigenvectors for $f$ in $V_1(\boldR),$ and let $\calC$ be the 
corresponding Hall basis.  
There is an ideal $\fraki$ isomorphic to $\fraki_0$ that is the
image of $\fraki_0$ under the isomorphism $g: \frakf_{n,r}(\boldR) \to \frakf_{n,r}(\boldR)$
defined by a bijection from $\calD$ to $\calC.$
The ideal $\fraki$ is invariant under
$f$ by Theorem \ref{actions}.  By Remark \ref{as cond 2}, 
the restriction of $f$ to $\fraki$
is unimodular. The third of the 
Auslander-Scheuneman condition holds by the theory of 
rational canonical forms. 
 The last condition is that all roots of modulus one or
minus one are in $\fraki:$ this holds because 
the set of 
roots of $p_1$ has full rank by Proposition \ref{full rank},
and  $\frakj_{n,r} < \fraki.$
Thus, $f$ descends to an Anosov automorphism of 
$\frakf_{n,r}(\boldR)/\fraki \cong \frakf_{n,r}(\boldR)/\fraki_0.$
\end{proof}

We can completely describe Anosov  Lie algebras whose associated 
polynomials $p_1$ are irreducible of prime degree $n \ge 3$ 
with cyclic Galois group. 
\begin{thm}\label{Cn}   
Suppose that 
$\frakn = \frakf_{n,r}(\boldR)/ \fraki$ is an $r$-step Anosov Lie 
algebra of type $(n_1, \ldots, n_r)$ 
admitting an Anosov automorphism defined by  
a semisimple hyperbolic matrix in $GL_n(\boldZ),$ 
rational Hall basis $\calB,$ the resulting automorphism  
$f$ in $\Aut(\frakf_{n,r}),$
and an ideal $\fraki$ satisfying Auslander-Scheuneman conditions.
  Let $(p_1, \ldots, p_r)$
be the associated $r$-tuple of polynomials. Suppose that $p_1$ 
is irreducible of prime degree 
 $n = n_1 \ge 3,$  and that $p_1$ has cyclic  Galois
group $G.$  
Then  the ideal $\fraki$ is of cyclic type as in Definition \ref{ideal-i}, and
$\fraki$ contains $\frakj_{n,r},$ where
  $\frakj_{n,r}$ is as defined in Definition \ref{j-def}.  
Furthermore, $n = n_1$ divides $n_i$ for all $i=2, \cdots, r.$

Conversely, for any prime $n \ge 3,$ a Lie algebra 
$\frakf_{n,r}(\boldR)/\fraki$ 
of cyclic type is Anosov,  as long as the ideal $\fraki$ contains
 $\frakj_{n,r}.$ 
\end{thm}

\begin{proof} 
Let $\fraki, \frakf_{n,r}$ and $(p_1,\ldots, p_r)$
  be as in the statement of the theorem.  
By Remark \ref{j in i}, the ideal
$\frakj_{n,r}$ is contained in $\fraki.$ 
 Recall that any irreducible polynomial in $\boldZ[x]$
of prime degree $n \ge 3$
 and cyclic Galois group has totally real roots; hence $f$ 
 has real spectrum.

Let $\calC_1$ be a basis of eigenvectors for $f|_{V_1(\boldR)}$ that is 
compatible with the rational structure, and let
$\calC = \cup_{i=1}^r \calC_i$ be the Hall basis defined by 
$\calC_1.$ 
 By Theorem \ref{actions}, for any $\bfw$ in $\calC_i,$ 
the orbit  
\[E_G^K(\bfw)= \myspan_\boldR \{ \sigma \cdot \bfw \, : \, \sigma \in G \}\] 
is a
 rational invariant subspace whose
characteristic polynomial  is a power $r^s$ of
an irreducible polynomial $r.$  The 
dimension $d$ of 
$E_G^K(\bfw) $
  is  $n$ or less because $|G| = n.$ 
The dimension $d$ also satisfies $d = s \cdot \deg r.$  
Because the splitting field for $r$ is a subfield of $\boldQ(p_1)$
and the Galois group $G$ for $p_1$ is isomorphic to the simple group 
$C_n,$ either 
$r$ is linear or $r$ is irreducible of degree $n$ and has Galois group $G.$
 Hence, either 
 $\fraki(G,\bfw)$ is  contained in $\fraki,$ or it is $n$-dimensional
and its intersection with $\fraki$ is trivial. 
The ideal $\fraki$ is then a direct sum of subspaces 
of form  $E_G^K(\bfw),$ hence is of cyclic type. 
 Since 
each step $V_i$ decomposes as the direct sum of $\fraki \cap V_i$
and $n$-dimensional  subspaces 
of the form $E_G^K(\bfw),$ for $\bfw \in \calC_i,$ 
 the dimension of the $i$th step of the quotient
$\frakn$ is divisible by $n,$ for all $i =2,\ldots,n.$ 

Now let $\frakn$ be the quotient
of $\frakf_{n,r}(\boldR) = \oplus_{i=1}^r V_i$  by an ideal
$\fraki_0$ of cyclic type relative to some Hall basis
$\calD = \cup_{i=1}^r \calD_i$, 
where $n \ge 3$ is prime, and suppose that 
$\fraki_0 > \frakj_{n,r}$ where $\frakj_{n,r}$ is defined relative to $\calD$.
  We will show that $\frakn$ is Anosov.   By 
Proposition \ref{existence}, there exists an Anosov polynomial 
$p_1$ whose Galois group is  cyclic of order $n.$  By using the companion 
matrix to $p_1$ and a Hall basis $\calB$
 we can define an automorphism $f$ of $\frakf_{n,r}(\boldR).$
Let $\calC_1$ be an eigenvector basis for $V_1$  that is compatible with
the rational structure defined by $\calB,$ and let $\calC$ be the 
Hall basis defined by $\calC_1.$  Then there is an ideal $\fraki$
of  $\frakf_{n,r}(\boldR)$ that is cyclic relative to the $G$ action on the Hall basis 
$\calC$ and is  
isomorphic to $\fraki_0.$ 

The ideal $\fraki$ is rational and invariant by Theorem \ref{actions}.
All we need to show is that the quotient map $\overline{f}$
 on $\frakf_{n,r}(\boldR)/\fraki$
 has no roots of modulus one.  But we have already shown in the first
part of the proof that if a rational
invariant subspace $\fraki(\bfw,G)$ is not in $\fraki,$ then
the characteristic polynomial $r$ for the restriction of $f$ to 
   $\fraki(\bfw,G)$ has odd degree $n.$  By Remark \ref{roots of modulus one},
$r$ has no roots of modulus one.  Hence  $\overline{f}$ is Anosov.
\end{proof}

\section{Anosov automorphisms in low dimensions}
\label{2&3}

In this section we describe Anosov automorphisms of some 
nilpotent Lie algebras that arise from  Anosov polynomials of low degree.  

\subsection{When $p_1$ is a product of quadratics}
We will analyze  Anosov automorphisms for which the associated polynomial
$p_1$ is a product of quadratic polynomials.   To do this  we 
need to define a family of two-step nilpotent Lie algebras.

\begin{definition}\label{graph} Let $\frakf_{n,2}(\boldR) = V_1(\boldR) \oplus V_2(\boldR)$  be the
free two-step Lie algebra on $2n$ generators, where $n \ge 2.$  Let  $\calC_1 =
\{\bfz_1,  \ldots \bfz_{2n}\}$  be a set of generating vectors spanning $V_1(\boldR),$
and let $\calC$ be the Hall basis  determined by $\calC_1.$
Let $S_1$ and $S_2$ be  subsets of the set  $\{ \{i,j\} \, : 1 \le i < j \le n\}$ of  subsets of 
$\{1, 2, \ldots, n\}$ of cardinality two.  
To the subsets $S_1$ and $S_2$ associate the ideal
$\fraki(S_1, S_2)$ of $\frakf_{n,2}(\boldR)$ defined by
\begin{multline} 
\fraki(S_1, S_2) = \bigoplus_{i=1}^n \myspan \{ [\bfz_{2i-1},\bfz_{2i}] \}
 \oplus \bigoplus_{\{i,j\} \in S_1} \myspan \{ [\bfz_{2i-1}, \bfz_{2j-1}], 
[\bfz_{2i}, \bfz_{2j}]\}
\oplus \\ 
\bigoplus_{\{i,j\} \in S_2} \myspan \{ [\bfz_{2i-1}, \bfz_{2j}], 
[\bfz_{2i}, \bfz_{2j-1}]\} \end{multline}
Define the two-step nilpotent Lie algebra $\frakn(S_1,S_2)$ to be
$\frakf_{2n,2}(\boldR)/ \fraki(S_1,S_2).$  A two-step Lie algebra
of this form will be said to be of {\em quadratic type}.
\end{definition}

In the next theorem we classify two-step Anosov Lie algebras such that 
the polynomial $p_1$ is a product of quadratics. 
\begin{thm}\label{2+2+...}  
Suppose that the polynomial $p_1$  associated to a semisimple
 Anosov automorphism
$f$ of a two-step
 Anosov Lie algebra $\frakn$ is the product of quadratic polynomials.  
Then $\frakn$ is of quadratic type, as defined in Definition \ref{graph},
and all the eigenvalues of $f$ are real. 
Furthermore, every two-step nilpotent  Lie algebra of quadratic type is Anosov.
\end{thm}
 
\begin{proof}  
 Let $\frakf_{2n,2}(\boldR) = V_1(\boldR) \oplus V_2(\boldR)$ be the free
two-step nilpotent Lie algebra on $2n$ generators.
Let $\overline{f}$ be an  an Anosov automorphism
of    $\frakn = \frakf_{2n,2}(\boldR)/\fraki$ defined by 
automorphism $f$ of $\frakf_{2n,2}(\boldR),$ a Hall basis $\calB$
and an ideal $\fraki$ satisfying the Auslander-Scheuneman conditions.  
 Without loss of generality,
assume that  $\fraki < V_2(\boldR).$ 
Let $(p_1,p_2)$ denote the pair of polynomials 
associated to $f,$ and assume that the   polynomial 
 $p_1$ of degree $2n$ is  the product 
of $n$ quadratic Anosov polynomials  $r_1, \ldots, r_n.$ 

 By the
quadratic equation, any roots of a quadratic Anosov polynomial are real.
All the eigenvalues of the Anosov automorphism $\overline{f}$
are Hall words in the roots of $p_1,$ hence are real.
 The subspace $V_1(\boldR)$ decomposes as the direct sum
$\oplus_{i=1}^n E_i$ of rational invariant subspaces such that
the characteristic polynomial for $f|_{E_i}$ is $r_i.$
For $i=1, \ldots, n,$  let $\bfz_{2i -1}$ and $\bfz_{2i}$  denote  eigenvectors
in $E_i$ with eigenvalues $\alpha_{2i -1}$ and $\alpha_{2i}$
respectively.  We may assume without loss of generality 
 that $\alpha_{2i -1} > 1 > \alpha_{2i} = \alpha_{2i-1}^{-1}.$ 
As in Example \ref{qr}, the polynomial  $p_2$ may be written as 
\[ p_2 = \prod_{i=1}^n (r_i \wedge r_i) \, \times  \prod_{1 \le i < j \le n}
(r_i \wedge r_j), \] 
and this factorization is corresponds to a decomposition
$V_2(\boldR) = \oplus_{1 \le i \le  j \le n}[E_i,E_j]$ of $V_2(\boldR)$ 
into rational invariant subspaces. 

 For all $i=1, \ldots, n,$ the  polynomial
 $r_i \wedge r_i$ is linear with 
 root $\alpha_{2i-1} \alpha_{2i} =  1,$ so
$(r_i \wedge r_i)(x) = x - 1.$  Therefore, the 
ideal $\fraki$ must contain the $n$-dimensional subspace
 $\oplus_{i=1}^n [E_i,E_i]$ of $V_2(\boldR).$ 
 For $i \ne j,$
the polynomial $r_i \wedge r_j$ is given by
\[ (r_i \wedge r_j)(x) = (x -\alpha_{2i-1}\alpha_{2j-1})
(x-\alpha_{2i-1}
\alpha_{2j-1}^{-1})(x-\alpha_{2i-1}^{-1}\alpha_{2j-1})
(x-\alpha_{2i-1}^{-1}\alpha_{2j-1}^{-1}).  \]
Minimal nontrivial invariant subspaces of 
 $[E_i,E_j]$ correspond to factorizations of $r_i \wedge r_j$
over $\boldZ.$

 If the splitting fields  $\boldQ(r_i)$ and $\boldQ(r_j)$
do not coincide, then they are  linearly disjoint, 
and $\boldQ(r_i \wedge r_j) = \boldQ(r_i)\boldQ(r_j)$
 is a biquadratic extension of $\boldQ.$
Therefore, $r_i \wedge r_j$ is irreducible, and $[E_i,E_j]$
has no nontrivial rational invariant subspaces.

Now suppose that the splitting fields
$\boldQ(r_i)$ and $\boldQ(r_j)$ are equal, for $i \ne j$.
Since the roots of Anosov quadratics are real, the field $\boldQ(r_i)$ is 
a totally real quadratic extension of $\boldQ.$ 
By Dirichlet's Fundamental  Theorem, there are units $\zeta = \pm 1,$
  and $\eta,$ a fundamental unit in $\boldQ(r_i)$, such 
that  any unit $\beta$ in $\boldQ(r_i)$ can be expressed
as  $\beta = \zeta^a \eta^{b},$
 where $a \in \{0, 1\}$ and $b \in \boldZ.$  
We may choose $\eta > 1.$
The Galois group for $r_1$ is generated by the automorphism of
$\boldQ(\eta)$ mapping $\eta$ to $\eta^{-1}.$

We can write 
\[ \alpha_{2i-1} =  \eta^{b_i},  \alpha_{2i} =   \eta^{-b_i},
\alpha_{2j-1} =   \eta^{b_j}, \text{and} \quad \alpha_{2j} =
 \eta^{-b_j}, \] where $b_i$ and $b_j$ are in $\boldZ^+.$
The four roots
of $ (r_i \wedge r_j)(x)$ are then the   numbers $\eta^{\pm b_i \pm  b_j}.$
Therefore, $r_i \wedge r_j$ factors over $\boldZ[x]$ as the product of
two  quadratics in $\boldZ[x]$
\begin{align*} 
 (x-\eta^{b_i + b_j})(x-\eta^{-b_i - b_j}) &= (x-\alpha_{2i-1}\alpha_{2j-1})
(x-\alpha_{2i}\alpha_{2j}), \quad \text{and} \\
(x-\eta^{b_i - b_j})(x-\eta^{-b_i + b_j}) &= (x-\alpha_{2i-1}\alpha_{2j-1}^{-1})(x-\alpha_{2i-1}^{-1}\alpha_{2j-1}^{-1}).
 \end{align*}
The first polynomial is irreducible since $b_i + b_j > 0,$ and $\eta$
is not a root of unity.  
If  $\alpha_{2i-1} = \alpha_{2j-1},$ 
then $b_i = b_j$ and the second polynomial is equal to  $(x-1)^2.$

This analysis shows that $\fraki \cap [E_i,E_j]$  
must be one of the subspaces $\{0\},$ $[E_i,E_j],$ 
\[ 
E_{i,j} = \myspan \{ [\bfz_{2i-1},\bfz_{2j-1}], [\bfz_{2i}, \bfz_{2j}]\} 
\quad \text{and} \quad 
E_{i,j}^\prime = \myspan  \{ [\bfz_{2i-1},\bfz_{2j}], [\bfz_{2i}, \bfz_{2j-1}]\}. \]
Then $\fraki$ is the direct sum of $\oplus_{i=1}^n [E_i,E_i],$
subspaces of the form $E_{i,j}$ and subspaces of the form 
$E_{i,j}^{\prime}.$  Therefore, $\frakn$ is of quadratic type.

Conversely, we show that 
for any ideal  $\fraki$ of quadratic type in $\frakf_{2n,2}(\boldR),$  
the quotient $\frakn= \frakf_{2n,2}(\boldR)/\fraki$ is an Anosov Lie algebra. 
Suppose that  $\fraki(S_1,S_2)$ is an ideal as in the definition
of quadratic type.  Fix a fundamental unit $\eta$ in a totally real
 quadratic extension of $\boldQ.$  Let $b_1, \ldots, b_n$ be 
distinct positive integers.
For each $i=1, \ldots, n,$
define the polynomial $r_i$ in $\boldZ[x]$ by
\[ r_i(x) = (x - \eta^{b_i})(x - \eta^{-b_i})\]
and let $p_1 = r_1 \cdots r_n.$  Then  for all $i \ne j,$
the polynomial 
$p_i \wedge p_j$ factors over 
$\boldZ$ as the product of pairs of irreducible quadratic polynomials 
\[ (x- \eta^{b_i + b_j})(x - \eta^{-b_i - b_j}) \quad \text{and} \quad 
 (x- \eta^{b_i - b_j})(x - \eta^{-b_i + b_j}) \]
in $\boldZ[x],$ neither having roots of modulus one.  
The two factors give a two rational invariant subspace of form $E_{ij}$
and $E_{ij}^\prime.$ 
 Therefore the ideal  
\[ \fraki = \bigoplus_{i=1}^n [E_i,E_i] \oplus 
\bigoplus_{\{i,j\} \in S_1} E_{i,j} \oplus
 \bigoplus_{\{i,j\} \in S_2} E_{i,j}^\prime\]
satisfies the four Auslander-Scheuneman conditions and is of the
form $\fraki(S_1,S_2)$ with respect to the appropriate Hall basis
of $\frakf_{2n,2}(\boldR).$
This completes the proof of  the theorem. 
\end{proof}

\subsection{When $p_1$ is a cubic}

We can classify all Anosov Lie algebras
of type $(3, \ldots, n_r)$ with $r=2$ or $r=3. $

\renewcommand{\arraystretch}{2}
\begin{table}
\begin{tabular}{|l | l | l | l |}
\hline 
$\frakn = \frakf_{n,r}(\boldR)/\fraki$ 
 & type & ideal $\fraki$ & reference for definition of $\fraki$ \\
\hline 
\hline 
$\frakf_{3,2}$  & $(3,3)$ & $\{0\}$ &\\
\hline
\hline 
$\frakf_{4,2}$  & $(4,6)$ & $\{0\}$ &  \\
\hline 
$\frakf_{4,2}/\fraki$ 
 & $(4,4)$ &  $\fraki(V_4, [\bfz_2,\bfz_1])$
 &  Definition \ref{ideal-i} \\
\hline
$\frakf_{4,2}/\fraki \cong \frakh_3 \oplus \frakh_3$ 
& $(4,2)$ & $\fraki(C_4, [\bfz_2,\bfz_1])$ & 
Definition \ref{ideal-i}\\
\hline 
\hline 
$\frakf_{5,2}$ & $(5,10)$ & $\{ 0 \}$  & \\
\hline 
$\frakf_{5,2}/\fraki$  & $(5,9)$ & $\fraki_1$ & Definition \ref{2+3-i}\\
\hline
$\frakf_{5,2}/\fraki$ & $(5,6)$ & $\fraki_1 \oplus \fraki_2$  &  Definition 
\ref{2+3-i} \\
\hline 
$\frakf_{5,2}/\fraki$ & $(5,5)$ & $\fraki(C_5, [\bfz_2,\bfz_1])$  &
Definition \ref{ideal-i} \\
\hline 
$\frakf_{5,2}/\fraki$ & $(5,5)$ & $\fraki_2$  &
Example \ref{D10} \\
\hline 
$\frakf_{5,2}/\fraki \cong \boldR^2 \oplus \frakf_{3,2}$
  & $(5,3)$ &  $\fraki_1 \oplus \fraki_3$
 &  Definition \ref{2+3-i}  \\
\hline 
\end{tabular}
\bigskip
\caption{\label{r=2}
Two-step Anosov Lie algebras of type $(n_1,n_2)$ with $n_1 \le 5$}
\end{table}

\begin{thm}\label{n1=3}
  If $\frakn$ is a two-step Anosov Lie algebra
of type $(3, n_2),$    then  $\frakn \cong \frakf_{3,2}.$ 
If $\frakn$ is three-step Anosov Lie algebra of type 
$(3, n_2, n_3),$   then 
  $\frakn$ is isomorphic to the Anosov Lie algebra 
$\frakf_{3,3}/F_2$ of type $(3,3,6),$ where 
 $\fraki = F_2$ is as defined in Equation
\eqref{F defs}, or
 $\fraki_1 = F_2 \oplus \fraki(C_3, [[\bfz_2, \bfz_1],\bfz_2])$
of type $(3,3,3),$
 where $\fraki(C_3, [[\bfz_2,\bfz_1],\bfz_2])$ is as defined in
Definition \ref{ideal-i}.   
\end{thm}

\begin{proof}
 Suppose that $A$ is a semisimple hyperbolic 
matrix in $GL_n(\boldZ)$ with associated triple of polynomials 
 $(p_1,  p_2, p_3).$ Let $\calB_1$ be a generating set for 
$\frakf_{3,3}$ and let $\calB$ be the Hall basis determined by
$\calB.$  Let
$\alpha_1, \alpha_2,  \alpha_3$ denote the  roots of $p_1.$
  Since
$p_1$ can not have 1 or -1 as a root, it is irreducible over
$\boldZ,$ and the Galois group $G$ of
the splitting field of $p_1$ is either $C_3$ or $S_3.$  

As demonstrated in Example \ref{free-3,2}, the cubic polynomial 
$p_2$ is irreducible and Anosov, 
so there are no Anosov Lie algebras of type $(3,n_2)$
other than $\frakf_{3,2}.$

Let $\fraki$ be an ideal so that $f$ descends to an Anosov
automorphism of a three-step nilpotent Lie algebra $\frakf_{3,2}/\fraki.$ 
If $G$ is symmetric, then $\fraki = F_2$ by Theorem \ref{Sn}.
 If $G$ is cyclic, 
by Theorem \ref{Cn},  $\fraki$ is either $F_2$ or it is
an ideal of form
 $F_2 \oplus \fraki(C_3, \bfw),$ for $\bfw \in \calC_3^{\prime}.$

Each Lie algebra that is listed may be realized
by choosing appropriate Anosov polynomial from Table \ref{low-degree-list}
and using its companion matrix $A$ to define an automorphism of
$\frakf_{3,2}$ or $\frakf_{3,3}.$ When $r=3,$ by Lemma \ref{odd-good}, all 
vectors with eigenvalue $\pm 1$ will be in the kernel $F_2 = \frakj_{3,3}.$
\end{proof}

\subsection{When $p_1$ is a quartic}
Now we consider the case that $p_1$ is a quartic Anosov polynomial.
The next lemma is useful for understanding Anosov 
Lie algebras of type $(4,n_2).$  

\begin{lemma}\label{4-galois} Let $(p_1, p_2)$ be the
pair of polynomials associated to an irreducible 
Anosov polynomial $p_1$ of degree four.  
Let $G$ denote the Galois group of the splitting field for
$p_1.$   Then
\begin{enumerate}
\item{$G\cong S_4$ or $G \cong A_4$ if and only if $ p_2$ is irreducible.}
\item{$G \cong C_4$ or $G \cong D_8$ if and only if $ p_2$ has an irreducible
quartic factor. }
\item{$G \cong V_4$  if and only if $ p_2$ has  no  irreducible factors of
degree three or more.}
\end{enumerate}
Furthermore,  roots of $p_2$ come in reciprocal pairs $\beta$ and 
$\pm \beta^{-1}.$
\end{lemma}

\begin{proof} Let  $\alpha_1, \alpha_2, \alpha_3, \alpha_4$ denote the
four distinct roots of $p_1.$  Then the roots of $ p_2$ are the six
numbers $\alpha_i \alpha_j,$ where $1 \le i< j \le 4.$    
Because $\alpha_1\alpha_2\alpha_3\alpha_4 =  \pm 1,$
  roots of $p_2$ come in pairs such as $\alpha_1\alpha_2$
and $\alpha_3\alpha_4 =  \pm (\alpha_1\alpha_2)^{-1}.$ 

The resolvent cubic $r$ for $p_1$ has roots 
\[ \beta_{1}=\alpha_1 \alpha_2 + \alpha_3 \alpha_4, \quad \beta_2= \alpha_1
\alpha_3 + \alpha_2 \alpha_4, \quad \text{and}  \quad\beta_3 = \alpha_1
\alpha_4 + \alpha_2 \alpha_3.\] 
Recall that one of three things must occur: 
(1) that none of $\beta_1,\beta_2,\beta_3$ lies in
$\boldQ,$  $r$ is irreducible, and $G\cong S_4$ or $G \cong A_4,$ 
(2)  exactly  one
of $\beta_1,\beta_2,\beta_3$ lies in $\boldQ,$  $r$ is the product of
an irreducible quadratic and a linear factor, and $G\cong C_4$ or $G \cong D_8,$
and  (3)   $\beta_1,\beta_2,\beta_3$ all lie in $\boldQ,$   $r$
splits over $\boldQ,$ and $G\cong V_4.$ 

Since $\alpha_1\alpha_2\alpha_3\alpha_4 = \pm 1,$ 
 $\beta_1 = \alpha_1 \alpha_2 + \alpha_3 \alpha_4$ is in $\boldQ$ 
if and only if
$(x - \alpha_1 \alpha_2)(x - \alpha_3 \alpha_4)$ is a quadratic factor
of $p_2$ over $\boldQ.$   
 Factors that are the counterparts from $\beta_2$ and $\beta_3,$ 
\begin{equation}\label{factorsabove}
 (x - \alpha_1 \alpha_3)(x - \alpha_2 \alpha_4) \qquad \text{and} \qquad
 (x - \alpha_1 \alpha_4)(x - \alpha_2 \alpha_3),
\end{equation}
are  in $\boldQ[x]$ depending on whether $\beta_1$ and $\beta_2$
are in $\boldQ.$
Therefore, when all of $\beta_1, \beta_2, \beta_3$ lie in
$\boldQ,$ $p_2$ factors as the product of quadratics in
$\boldQ[x],$ establishing the claim in Case (3).

In Case (1),  $p_2$ is irreducible by Theorem \ref{Sn}, Part \eqref{2,2}.

  In the second case,  when $G$ is $C_4$ or $D_8,$  there is a  $G$ orbit
of cardinality four, something like
  $\{\alpha_1\alpha_3, \alpha_1\alpha_4,\alpha_2\alpha_3,
\alpha_2\alpha_4\},$ that by Theorem \ref{actions} yields
a factor
\[ r(x)=(x - \alpha_1 \alpha_3)(x - \alpha_2 \alpha_4)(x - \alpha_1
\alpha_4) (x - \alpha_2 \alpha_3) \]
of $p_2,$
and an orbit of cardinality two 
corresponding to a factor 
$(x - \alpha_1 \alpha_2)(x - \alpha_3 \alpha_4)$ of $p_2.$  
Then $\beta_1 \in \boldQ$ and  $\beta_2, \beta_3 \not \in \boldQ.$
   By Theorem \ref{actions},
if $r$ were to factor, it would be a power of an irreducible.  
Knowing that  $\beta_2, \beta_3 \not \in \boldQ$ rules out
the polynomials in Equation \eqref{factorsabove} as factors of $r.$
The only other
options for factors yield contradictions to the distinctness of roots of
$p_1.$  Thus, $r$ is an irreducible quartic factor of $p_2.$
\end{proof}

Now we are ready to
 classify two-step Anosov Lie algebras of type $(4,n_2)$ using
the methods we have established.  
\begin{thm}\label{n=4,r=2}  
If $\frakn$  is an Anosov Lie algebra  of type $(4, n_2),$ 
 then  $\frakn$ is one of  the Anosov Lie algebras listed in Table \ref{r=2}.   
\end{thm}

\begin{proof}
Suppose that $f$ is a semisimple automorphism of $\frakf_{4,2}(\boldR) 
= V_1(\boldR) \oplus V_2(\boldR)$ 
that projects to an Anosov automorphism of
 a two-step quotient $\frakn = \frakf_{4,2}/\fraki.$ 
 Let $(p_1,p_2)$ be the polynomials associated to 
$f,$ and let $K$ be the splitting field
for $p_1.$   Let $\alpha_1, \alpha_2, \alpha_3, \alpha_4$ be the roots
of $p_1$ and let $\bfz_1, \bfz_2, \bfz_3, \bfz_4$
be corresponding eigenvectors for $f_K$ in $V_1(K).$

If $p_1$ is reducible, since it is Anosov, it is a product of
quadratics, and Theorem \ref{2+2+...} describes the Anosov Lie
algebras in that situation: $\fraki$ is one of $\fraki(V_4,[\bfz_2,\bfz_1])$
and $\fraki(C_4,[\bfz_2,\bfz_1]).$ 

Assume that $p_1$ is
irreducible.
The Galois group $G$ of $p_1$ is then a transitive
permutation group of degree  four.  

By Lemma \ref{4-galois}, if $G$ is $S_4$ or $A_4,$ the polynomial
$p_2$ is irreducible, so $V_2(\boldR)$ has no nontrivial 
proper rational invariant subspaces; hence, 
 $\fraki = \{0\},$ and $\frakn$ is free. 
If $G$ is $D_8$ or $C_4,$ then by considering the
action of $G$ on the set of roots it can be seen 
that $\frakf_{4,2}(K)$ is the direct sum of a 
four-dimensional rational $f_K$-invariant 
subspace $\fraki_1$ and a two-dimensional rational $f_K$-invariant  subspace
 $\fraki_2.$  If the roots are real, these may be represented
as follows,  renumbering the roots if necessary, 
\[
\fraki_1 = \fraki(C_4,[\bfz_2,\bfz_1]) =  \fraki(D_8,[\bfz_2,\bfz_1]), \quad \text{and} \quad
\fraki_2 = \fraki(C_4,[\bfz_3,\bfz_1]) = \fraki(D_8,[\bfz_3,\bfz_1]). \]
By Lemma  \ref{4-galois}, the minimal polynomial for 
$\fraki_1$ is irreducible, so $\fraki$ is 
a minimal nontrivial invariant subspace.  If there are complex roots,
a short computation shows that the rational invariant 
subspaces $\fraki_1^\prime$ and $\fraki_2^\prime$ of $\frakf_{4,2}(K)$
yield rational invariant subspaces $(\fraki_1^\prime)^\boldR$ and 
 $(\fraki_2^\prime)^\boldR$ of $\frakf_{4,2}(\boldR)$ isomorphic 
to $\fraki_1$ and $\fraki_2.$
The characteristic polynomial for 
$\fraki_2$ is either irreducible or it has roots 
$\pm 1,$ in which case $\fraki_2$ must be
contained in the ideal $\fraki.$  Thus, when $G = C_4$
or $D_8,$  the ideal $\fraki$ is  $\{0\}, \fraki_1$
or $\fraki_2.$

Finally, if  $G$ is the Klein four-group,  by the same reasoning, 
$V_2(K)$ is the direct sum of three two-dimensional ideals that are
either  minimal or must be contained in $\fraki.$  Then $\frakn$
is of quadratic type.  

Anosov automorphisms of all types $(4,n_2)$  may be realized
by choosing an appropriate polynomial $p_1$ from Table \ref{low-degree-list}.
By Remark \ref{roots of modulus one}, the  polynomial $p_2$ defined by such
a $p_1$ can not have any nonreal roots of modulus one unless it is 
self-reciprocal.  
The only polynomials listed in Table \ref{low-degree-list}  that is 
self-reciprocal is for $V_4,$ so 
 $p_2$ will not have roots of modulus one unless $G = V_4,$ in which case
the eigenspaces of those roots lie in the ideal $\fraki.$ 
\end{proof}

\subsection{When $p_1$ is a quintic}

First we define some two-step nilpotent 
Lie algebras on five generators.
\begin{definition}\label{2+3-i}
Let $\frakf_{5,2} = V_1(\boldR) \oplus V_2(\boldR)$ be a free Lie algebra on five generators
$\{ \bfz_i \}_{i=1}^5.$  Define subspaces $E_1$ and $E_2$ of 
$V_1(\boldR)$ by  $E_1 = \myspan_{\boldR} \{ \bfz_1, \bfz_2 \}$
and let $E_2 = \myspan_{\boldR} \{\bfz_3, \bfz_4, \bfz_5\},$
and define  ideals of  $\frakf_{5,2}$  by 
\[
\fraki_1 =  [E_1,E_1], \quad
\fraki_2 =  [E_1,E_2], \quad \text{and} \quad
\fraki_3 =  [E_2,E_2]. \]
Define two-step  Lie algebras by
\[
\frakn_1 = \frakf_{5,2}/ \fraki_1, \quad
\frakn_2 = \frakf_{5,2}/ (\fraki_1 \oplus \fraki_2), \quad  \text{and} \quad
\frakn_3 = \frakf_{5,2}/ (\fraki_1 \oplus \fraki_3).
\]
These are 
of types $(5,9)$, $(5,3)$ and $(5,6)$ respectively.  Note that
$\frakn_2 \cong \boldR^2 \oplus \frakf_{3,2}.$ 
\end{definition}

\begin{thm}\label{n1=5}
  Suppose that $\frakn$ is a  two-step 
nilpotent Lie algebra 
of type $(5, n_2)$ admitting an Anosov automorphism $f.$ 
 Then $\frakn$ is one of the Lie algebras listed in Table \ref{r=2}.
  Furthermore, all of the Lie algebras of 
type  $(5, n_2)$ in Table  \ref{r=2} are Anosov.
\end{thm}

\begin{proof} 
Let $(p_1,p_2)$ be the pair of polynomials associated to an 
automorphism $f$ of $\frakf_{5,2} = V_1(\boldR) \oplus V_2(\boldR)$ that projects
to an Anosov automorphism of an Anosov Lie algebra  
$\frakn = \frakf_{5,2}/\fraki,$  
for some ideal $\fraki$ of $\frakf_{5,2}$ satisfying the
Auslander-Scheuneman conditions.   
 Without loss of generality we assume that the
roots of $p_1$ have product 1.  

First suppose that $p_1$ is irreducible, so that its Galois group $G$ is
isomorphic to $S_5, A_5, D_{10}, C_{5}$ or the holomorph $\Hol(C_5)$ of $C_5.$
If $G$ is isomorphic to one of 
$S_5, A_5$, and $\Hol(C_5),$ then the action of 
$G$ on the roots of $p_1$ is two-transitive, so
  by Theorem \ref{Sn}, 
$\frakn$ is isomorphic to  $\frakf_{5,2}.$  
The case that the Galois group is $D_{10}$ was considered in Example
\ref{D10}, where it was found that either $\frakn$ is free, 
 $\frakn$ is isomorphic to $\frakn_1$ or $\frakn_2$ in 
Example \ref{D10}, both of type
$(5,5).$  The  case of $C_5$ is covered by Theorem \ref{Cn}.
Thus, in all possible cases, $\frakn$ is one of the
Lie algebras listed in Table \ref{r=2}.   Each example may be realized
by choosing a polynomial $p_1$ from
 Table \ref{low-degree-list}, and using its companion matrix to 
define an automorphism of $\frakf_{5,2}.$ 
To get $\frakn_1,$ one needs to choose $p_1$ with all real
roots, and to get $\frakn_2,$ one needs $p_2$ to have 
four nonreal roots.  Examples of both kinds are in the table.
The associated polynomials
 $p_2$ have no roots of modulus one by Lemma \ref{odd-good}.

Now suppose that the Anosov polynomial $p_1$ is the product of a 
quadratic Anosov polynomial $r_1$
and a cubic Anosov polynomial $r_2$.  Let $E_1$ and $E_2$
 denote the rational invariant subspaces of $V_1(\boldR)$ corresponding
to $r_1$ and $r_2$ respectively.  
Because $r_1$ is quadratic, 
$ r_1 \wedge r_1  = x \pm 1,$ so $\fraki_1 = [E_1,E_1]$ must be
contained in $\fraki.$  As seen in Example \ref{free-3,2}, since
 $r_2$ is a cubic,  the polynomial 
$r_2 \wedge r_2 $ is irreducible and Anosov.
Therefore, the set $\fraki_3 = [E_2,E_2]$ is a minimal nontrivial invariant subspace of $V_2(\boldR).$ 

Let $\alpha_1$ and $\alpha_2$ denote the roots of $r_1,$ while
$\alpha_3, \alpha_4, \alpha_5$ are the roots of $r_2.$  Then 
$\alpha_1\alpha_3$ is a root of 
$r_1 \wedge r_2.$  By standard arguments, $[\boldQ(\alpha_1\alpha_3): \boldQ] = 6,$  so $r_1 \wedge r_2$ is the minimal polynomial of $\alpha_1\alpha_3.$
Therefore  $r_1 \wedge r_2$ is
irreducible and $\fraki_2 = [E_1, E_2]$ is a minimal nontrivial
rational invariant subspace.
The subspace
$V_2(\boldR)$ decomposes as the sum $\fraki_1 \oplus \fraki_2 \oplus \fraki_3$
of minimal nontrivial invariant subspaces, where
$\fraki_1, \fraki_2$ and $\fraki_3$ are as in Definition \ref{2+3-i},
and the only possibilities for an ideal $\fraki$ defining a 
two-step Anosov quotient are $\fraki_1, \fraki_1 \oplus \fraki_2$ and
$\fraki_1 \oplus \fraki_3,$ as claimed. 

Choosing $r_1$
and $r_2$ to be arbitrary  Anosov polynomials of degree two and three
respectively will yield an Anosov polynomial $p_1 = r_1 r_2$ such that
 the corresponding
automorphism $f$ of $\frakf_{5,2}$ admits quotients of all types  
listed in the table.  By choosing $r_2$ so it has real roots,
  the roots of $p_2$ will be real, and 
$r_1 \wedge r_2$ will have no roots of modulus one. 
\end{proof}

\section{Proofs of main theorems}\label{summary}

Now we provide proofs for the theorems presented in Section \ref{introduction}.

\begin{proof}[Proof of Theorem \ref{main}.] 
Suppose that $\frakn$ is a two-step Anosov Lie algebra of type $(n_1,n_2)$
with associated polynomials $(p_1,p_2).$
If $n_1 = 3, 4,$ or $5,$ then $\frakn$ is one of the Lie algebras in
 Table \ref{r=2}, by Theorems \ref{n1=3}, \ref{n=4,r=2}
and \ref{n1=5}.   Therefore, 
Part \eqref{classify-lowdim} of Theorem \ref{main}  holds.

The second part follows immediately from by Part \eqref{2,2} of 
Theorem \ref{Sn}. 
\end{proof}

\begin{proof}[Proof of Theorem \ref{Sn-Cn}.] 
The first part of the theorem follows immediately from Theorem \ref{Cn}.
The second part is a consequence of  Theorem \ref{Sn}.
\end{proof}

\begin{proof}[Proof of Theorem \ref{general-properties}.]
Corollaries \ref{prime-dim} and \ref{dimensions} imply
the theorem.  
\end{proof}

\begin{proof}[Proof of Theorem \ref{spectrum}.]
Suppose that the spectrum of $f$ is in $\boldQ(\sqrt{b}).$  Then 
the polynomial $p_1$ associated to $f$ 
is a product of quadratics, each of whose roots lie in $\boldQ(\sqrt{b}).$
Theorem \ref{2+2+...} implies that $\frakn$ is one of the 
Lie algebras defined in Definition \ref{graph}.
\end{proof}

\bibliographystyle{alpha}
\bibliography{anosovbib}
\end{document}